\documentclass[11pt]{amsart}
\usepackage{indentfirst,latexsym,bm}
\usepackage{amsfonts}
\usepackage{amssymb}
\usepackage{amsmath}
\usepackage{hyperref}
\usepackage{dsfont}
\usepackage{leftidx}
\usepackage{amsthm}
\allowdisplaybreaks[4]
\usepackage{geometry}
\usepackage{amsmath,amscd}
\usepackage{shuffle}
\usepackage[all,cmtip]{xy}
\usepackage{color,xcolor}
\usepackage[all]{xy}

\newtheorem{theorem}{Theorem}[section]
\newtheorem{proposition}[theorem]{Proposition}
\newtheorem{lemma}[theorem]{Lemma}
\newtheorem{corollary}[theorem]{Corollary}
\theoremstyle{definition}
\newtheorem{definition}[theorem]{Definition}
\newtheorem{example}[theorem]{Example}
\newtheorem{remark}[theorem]{Remark}

\newcommand{\K}{\mathds{k}}

\newcommand{\Z}{\mathbb{Z}}
\newcommand{\X}{\langle X \rangle}
\newcommand{\st}{\text{st}}
\newcommand{\Char}{\operatorname{char}}
\newcommand{\id}{\operatorname{id}}

\theoremstyle{plain}
\newcounter{maint}

\theoremstyle{plain}

\begin{document}
 \thispagestyle{empty}

\title[Noncommutative binomial theorem, shuffle type polynomials and Bell polynomials]
{Noncommutative binomial theorem, shuffle type polynomials and Bell polynomials}

\author{
{Huan Jia $^{1,2}$}
and 
{Yinhuo Zhang $^{1,*}$}}

\address{$^{1}$ Department of Mathematics and Statistics,
University of Hasselt, Universitaire Campus, 3590 Diepenbeek, Belgium}

\address{$^{2}$ School of Mathematical Sciences, East China Normal University, Shanghai 200241, China}

\address{Huan Jia $^{1,2}$}
\email{huan.jia@uhasselt.be}

\address{Yinhuo Zhang $^{1,*}$}
\email{yinhuo.zhang@uhasselt.be}

\date{}
\maketitle

\begin{abstract}
In this paper we use the Lyndon-Shirshov  basis to  study the shuffle type polynomials. We give a free noncommutative binomial (or multinomial) theorem in terms of the Lyndon-Shirshov basis. Another noncommutative binomial theorem given by the shuffle type polynomials with respect to  an adjoint  derivation is established. As a result, the Bell differential polynomials and the $q$-Bell differential polynomials can be derived from the second binomial theorem. The relation between the shuffle type polynomials and the  Bell  differential polynomials is established.  Finally, we give some applications of the free noncommutative binomial theorem including application of the shuffle type polynomials to bialgebras and Hopf algebras. 

\bigskip

\end{abstract}

\section*{Introduction}

The binomial theorem is a fundamental and perhaps most well-known result in Algebra that describe the expansion of a (non-negative integer) power of a binomial.  The theorem is widely used in mathematics, physics, and engineering, and has many important applications such as in combinatorics, probability, statistics and optimizations, etc. The binomial coefficients are extremely important as they are everywhere in discrete mathematics. The first formulation of the full theorem including a proof of it came in the 10 century by the Persian mathematician Al-Karaji. Since then, there  have been a number of generalizations of the theorem, one of which is attributed by Newton, from the non-negative integer power to a real number power. Another generalization worth mentioning  is the so-called multinomial theorem of Leibniz, which considers the expansion of a general multinomial $(x_1+x_2+\ldots +x_m)^n$ into a polynomial of $m$ variables. This result has found numerous applications in the field of combinatorics.  

The other direction of generalization is to consider the noncommutative variables and their multinomial theorem.   The earliest noncommutative  version of  the binomial theorem perhaps appeared in  \cite{Y1952}, where  the author gave the binomial formula  for elements in the Heisenberg-Weyl algebra $A_{h}$ generated by $x,y$ satisfying the relations: $xy-yx=h$, $xh=hx$ and $hy=yh$. The noncommutative binomial theorem reads  as follows:
\begin{align}\label{formula-1}
(x+y)^{n}=\sum_{j=0}^{[n/2]} \sum_{i=0}^{n-2j} \frac{n !}{j! i!(n-2j-i)!}\left(\frac{h}{2}\right)^{j}y^{i} x^{n-2j-i}.
\end{align}
This formula was also considered by Cohen \cite{C1966} and Wilcox \cite{W1967}, and was generalized by Mikha\u{\i}lov \cite{M1983}.

Another important $h$-analogue of the binomial formula was established by Benaoum in \cite{B1998} for the elements in the $h$-deformed quantum plane (or Jordan quantum plane) with the relation: $xy=yx+hy^{2}$ where $h$ is a central element.  Benaoum in \cite{B1999} obtained the binomial theorem in  case  $xy=qyx+hy^{2}$. Benaoum's binomial formula was later generalized by Mansour--Schork  to the case: $xy=qyx+hf(y)$ where $f(y)$ is a polynomial \cite{MS2011}.

Apart from the $h$-analogue of the binomial formula, there are $q$-analogues of the binomial formula. The first $q$-analogue 
appeared in \cite{Sch1953}.  In the quantum plane $\mathbb{C}_q[x,y]$, where $xy=qyx$,  the $q$-binomial formula is valid:
\begin{equation}\label{q-binomial}
(x+y)^n=\sum_{j=0}^n {\binom{n}{j}}_q y^{n-j}x^j
\end{equation}
where ${\binom{n}{j}}_q $ is the $q$-binomial (or Gaussian binomial) coefficient.  Since then, there have been various generalizations of the $q$-binomial formula  (\ref{q-binomial}). We mention here a combined one of  the formulas (\ref{formula-1}) and (\ref{q-binomial}) given by Blumen \cite{B2006}.  
\begin{align}\label{formula:Blumen}
(x+y)^{n}=\sum_{r+2s+t=n} \frac{n!}{(r)_{q}!(2)_{q}^{s}(s)_{q^{2}}!(t)_{q}!} y^{r}h^{s}x^{t}.
\end{align}
This formula is obtained from the relations: $xy=qyx+h$, $xh=q^{2}hx$, $hy=q^{2}yh$ and $q\neq 0$,   satisfied by the coproducts of non-simple root vectors in the quantum superalgebra $U_{q}(osp(1|2n))$.  When $q=1$,  Formula (\ref{formula:Blumen}) becomes Formula (\ref{formula-1}). When $h=0$, Formula (\ref{formula:Blumen}) is the $q$-binomial formula (\ref{q-binomial}). 

The noncommutative binomial formulas mentioned above have strong relation with the so-called Bell differential polynomials. 
 In 1998  Schimming--Rida \cite{SR1998} introduced the (noncommutative) Bell differential polynomials to give the binomial theorem for any associative algebra or free algebra, see Theorem \ref{binomial formula of Schimming-Rida}. At the same time,  Munthe-Kaas \cite{MK1995,MK1998} introduced certain noncommutative polynomials in numerical integration on manifolds, which are essentially the noncommutative Bell polynomials \cite{SR1996}. Furthermore,   Lundervold--Munthe-Kaas \cite{LM2011} used the noncommutative Bell polynomials to give the binomial formula (\ref{formula:noncomm-binom-LMK}).  

The  Bell  differential polynomials are closely related to  the shuffle type polynomials. In this paper,  we study the shuffle type polynomials, and use the Lyndon-Shirshov basis to  give a  noncommutative binomial formula for two free variables.    Using the adjoint derivation and  the shuffle type polynomials  we establish  another noncommutative binomial theorem,  from which the Bell differential polynomials and the $q$-Bell differential polynomials can be deduced.   Applications of the free noncommutative binomial theorem including application of the shuffle type polynomials to bialgebras and Hopf algebras will be given. 

The paper is organized as follows. In Section 1, we recall some results of Lyndon words and Lyndon-Shirshov basis.  In Section \ref{Section 2}, we establish a free noncommutative binomial formula using the Lyndon-Shirshov basis, see Corollary \ref{cor:noncomm-binom-thm-LSbasis}. The noncommutative binomial coefficients are determined by the following theorem. 

\textbf{Theorem A.} [Theorem~\ref{thm:noncom-binom-coeff-LSbasis}]
\textit{ Let $(X,\preceq)=(\{1,2\},1\prec 2)$.
Suppose that $E_{\alpha_{1}}^{t_{\alpha_{1}}} \cdots E_{\alpha_{n}}^{t_{\alpha_{n}}}$ is a term of $\mathcal{SH}_{i,j}(E_{2},E_{1})$, where $\alpha_{k}\in \mathcal{L}_{X}$, $t_{\alpha_{k}}\in \mathbb{N}$, $1\le k\le n $, $\alpha_{1}\succ \ldots\succ \alpha_{n}$, $i=\sum_{k=0}^{n}t_{\alpha_{k}}\alpha_{k}(2)$ and $j=\sum_{k=0}^{n}t_{\alpha_{k}}\alpha_{k}(1)$. Then the coefficient of $E_{\alpha_{1}}^{t_{\alpha_{1}}}\cdots E_{\alpha_{n}}^{t_{\alpha_{n}}}$ in $\mathcal{SH}_{i,j}(E_{2},E_{1})$ is given by
\begin{equation*}
C_{E_{\alpha_{1}}^{t_{\alpha_{1}}} \cdots E_{\alpha_{n}}^{t_{\alpha_{n}}}}=\frac{(i+j)!}{\prod_{k=1}^{n}(|\alpha_{k}|!)^{t_{\alpha_{k}}}t_{\alpha_{k}}!}\prod_{k=1}^{n}C_{E_{\alpha_{k}}}^{t_{\alpha_{k}}},
\end{equation*}
where $i+j=\sum_{k=1}^{n}|\alpha_{k}|t_{\alpha_{k}}$ and $C_{E_{\alpha_{k}}}$ is the coefficient of the term $E_{\alpha_{k}}$ in $\mathcal{SH}_{\alpha_{k}(2),\alpha_{k}(1)}(E_{2},E_{1})$.
}

Theorem A holds as well in  the multinomial case, see Corollary \ref{cor:noncom-multinomial-thm}. 
In some special cases, we obtain useful results from Theorem A, see Corollary \ref{cor:particular-case-noncom-binom}.

 Inspired by the work of Mansour--Schork \cite[Research problem 8.1]{MS2016}, Schimming--Rida \cite{SR1996,SR1998} and Lundervold--Munthe-Kaas \cite{LM2011}, we obtain the following noncommutative binomial formula given by the shuffle type polynomials related to an adjoint $\sigma$-derivation $\operatorname{ad}_{\sigma}x$ in Section \ref{Section 3}. 

\textbf{Theorem B.} [Theorem~\ref{thm:binom-formula-shuffle-poly-sigma-derivation}]
\textit{
Let $A$ be an associative algebra and $x,y\in A$. Suppose that there is an algebra homomorphism $\sigma: A\rightarrow A$. Then the following holds for $n\in \mathbb{N}$
\begin{align*}
(x+y)^{n}=\sum_{k=0}^{n}\widehat{\mathcal{SH}}_{k,n-k}(1)x^{n-k},
\end{align*}
where $\widehat{\mathcal{SH}}_{k,n-k}:=\mathcal{SH}_{k,n-k}(\operatorname{ad}_{\sigma}x+y,\sigma)$ and $\operatorname{ad}_{\sigma}x(a):=xa-\sigma(a)x$ for $a\in A$.
 }

In some special cases, Theorem B derives the Bell differential polynomials (see Remark \ref{rmk:SHn,n-k(1)-to-Bell-diff-poly}),  the $q$-Bell differential polynomials (see Section \ref{Section 3.2}) and the corresponding noncommutative binomial formula (see Corollary \ref{cor:binom-formula-by-q-Bell-diff-poly}). Note that the $q$-Bell differential polynomials are essentially the noncommutative versions of the $q$-Bell polynomials introduced by Johnson \cite{J1996}, see Remark \ref{rmk:noncom-version-of-q-Bell-poly}.

In Section \ref{Section 4.1}, we recall the Bell differential polynomials  or the noncommutative Bell polynomials.
In Section \ref{Section 4.2}, we give the following relation between the shuffle type polynomials and the Bell differential polynomials:

\textbf{Theorem C.} [Theorem~\ref{bell-shuffle}]
\textit{
Let $(X,\preceq)=(\{1,2\},1\prec 2)$ and $\mathcal{SH}_{i,j}:=\mathcal{SH}_{i,j}(E_{2},E_{1})$ for $i,j\ge 0$. \\
$\mathrm(a)$ Let $\mathcal{SH}_{i,j}^{[2]}$ be the sum of all the terms in $\mathcal{SH}_{i,j}$ whose rightmost element is not $E_{1}$. Then $\widehat{B}_{n,k}(E_{1},E_{2})=\mathcal{SH}_{k,n-k}^{[2]}$.\\
$\mathrm(b)$ Dually, denote by $\mathcal{SH}_{i,j}'$ the sum of all the terms in $\mathcal{SH}_{i,j}$ whose leftmost element is not $E_{2}$. Then $\widehat{B}_{n,k}^{*}(E_{2},E_{1})=\mathcal{SH}_{n-k,k}'$.
}

Thus Theorem A and Corollary \ref{cor:noncomm-binom-thm-LSbasis} can be used to express the Bell differential polynomials in terms of the Lyndon-Shirshov basis, see Corollary \ref{cor:Bell-diff-poly-by-LS-basis}.

In Section \ref{Section 5.1}, we give some applications of Theorem A, which contain some classical binomial formulas and Formula (\ref{formula-1}) (see Remark \ref{rmk:quasi-comm-case}). In Section \ref{Section 5.2}, we use the $q$-Bell differential polynomials to deduce the $q$-commutative Bell differential polynomials (see Theorem \ref{thm:q-comm-Bell-diff-poly}) and the corresponding binomial formula (see Corollary \ref{cor:binom-q-comm-Bell-diff-poly}). The binomial formula generalizes Blumen's formula (\ref{formula:Blumen}), see Remark \ref{rmk:q-Bell-diff-poly-to-Blumen-formula}. In Section \ref{Section 5.3}, we show that the shuffle type polynomials are strongly related the polynomials $Q_{m}^{(n)}$ which determine the coproduct of the noncommutative Fa\`a di Bruno bialgebra \cite{BFK2006}, see Theorem \ref{thm:shuffle-poly-BFdBnc-coproduct}. We mention that Radford \cite{R1979} used the shuffle type polynomials to construct some $\mathbb{N}$-graded pointed bialgebras or Hopf algebras, see Proposition \ref{pro:bialgebra-shuffle-relations}.

\section{Preliminaries} 

\subsection{Lyndon-Shirshov basis and shuffle type polynomials}\qquad

Let $\K$ be a field. Following Lothaire \cite{L1997}, we recall some results of Lyndon-Shirshov words and basis. Let $(X,\prec)$ be a totally ordered set. $\langle X \rangle$ is the free monoid generated by $X$. The order $\prec$ induces
the \textit{lexicographic order} $\prec$ on $\X$ defined as follow:

for words $\alpha,\beta\in \X$,  
\begin{align*}
\alpha\prec \beta \Longleftrightarrow \left\{\begin{array}{l}
\alpha \text { is a proper left factor of } \beta, \text { or } \\
\alpha=\gamma x \eta \ s, \beta=\gamma y \zeta \text { with } x, y \in X, x\prec y \text{ and } \gamma, \eta, \zeta \in\langle X\rangle.
\end{array}\right.
\end{align*}
Note that the order $\alpha\prec \beta$ defined above is strict. 
For $\alpha\in \langle X \rangle$, if $\alpha=x_{1}x_{2}\ldots x_{n},~ x_{i}\in X$, the \textit{length} of $\alpha$ is $n$,  
denoted $|\alpha|$. For a letter $y\in X$, denote by $\alpha(y):=\# \{x_{i}|~x_{i}=y,~1\le i\le n\}$. The empty word is denoted by $1_{X}$.  We write  $\X_{0}$ for  $\X \setminus \{1_{X}\}$.  $\K\langle X\rangle$ is the free associative algebra over $\K$ generated by $X$.

\begin{definition}
A word $\alpha$ in $\X_{0}$ is a \textit{Lyndon word} (or \textit{Lyndon-Shirshov word}) if $\alpha\in X$, or  $\alpha=\beta\gamma$ such that $\alpha\prec \gamma\beta$ for $\beta,\gamma\in \X_{0}$. Denote by  $\mathcal{L}$ or $\mathcal{L}_{X}$  the set of all Lyndon words in $\langle X \rangle$. 
\end{definition}

The Lyndon-Shirshov words were independently introduced by Lyndon \cite{L1954,L1955} and Shirshov \cite{S1953}.

\begin{example}
Let $(X,\preceq)=(\{1,2\},1\prec 2)$.  The list of the first Lyndon-Shirshov words is $$\mathcal{L}=\{1,2,12,112,122,1112,1122,1222,11112,11122,11212,11222,12122,12222,\ldots\},$$
where
$1\prec 11112 \prec 1112 \prec 11122 \prec 112 \prec 11212 \prec 1122 \prec 11222  \prec 12 \prec 12122 \prec 122 \prec 1222 \prec 12222 \prec2$.
\end{example}

The following lemma shows that the factorization of a Lyndon word is not unique.

\begin{lemma}{\cite[Proposition 5.1.3]{L1997}}\label{Lemma 2.3}
Let $\alpha\in \langle X \rangle$. $\alpha\in \mathcal{L}$ if and only if $\alpha\in X$ or $\alpha=\beta\gamma$ such that $\beta,\gamma\in \mathcal{L}$ and $\beta\prec \gamma$. More precisely, if $\alpha=\beta\gamma\in \mathcal{L}$ such that $\gamma$ is the proper right factor of maximal length, then $\beta\in \mathcal{L}$ and $\beta\prec \alpha\prec \gamma$. 
\end{lemma}

For a composed word $\alpha=\beta\gamma\in \mathcal{L}\setminus X, \beta,\gamma\in \mathcal{L}$, the pair $(\beta,\gamma)$ is called the \textit{standard factorization} of $\alpha$ if $\gamma$ is of maximal length. Denote it by $\operatorname{st}(\alpha)$ or $(\alpha_{L},\alpha_{R})$.

The following lemma shows that every word can be uniquely decomposed as a product of Lyndon words in the non-increasing order.

\begin{lemma}{\cite[Theorem 5.1.5]{L1997}}\label{Lemma 2.4}
Every word $\alpha \in \X_{0}$ can be written uniquely as a non-increasing product of Lyndon words, that is,
\begin{align*}
\alpha=\beta_{1} \beta_{2} \cdots \beta_{n},
\end{align*}
where $\beta_{i} \in \mathcal{L}$, $1\le i\le n$ and $\beta_{1} \succeq \beta_{2} \succeq \ldots \succeq\beta_{n}$.
\end{lemma}

The Lyndon-Shirshov basis was introduced by Chen--Fox--Lyndon \cite{CFL1958} and Shirshov \cite{S1958} for a basis of the free Lie algebra.

\begin{lemma}{\cite[Theorem 5.3.1]{L1997}}\label{Lemma 2.5}
Let $\mathfrak{L}(X)$ be the free Lie algebra over $\K$ generated by $X$ and define a map $E:\mathcal{L}\longrightarrow \mathfrak{L}(X)$ inductively by $E(x)=x$ if $x\in X$; and if $\alpha\in \mathcal{L}\setminus X$ and $\operatorname{st}(\alpha)=(\beta,\gamma)$, by
\begin{align*}
E(\alpha)=[E(\beta),E(\gamma)]:=E(\beta)E(\gamma)-E(\gamma)E(\beta).
\end{align*} 
Then $\mathfrak{L}(X)$ is a free $\K$-module with $E(\mathcal{L})$ as a basis, called the \textit{Lyndon-Shirshov basis}. Denote $E(\alpha)$ by $E_{\alpha}$ for convenience in computations.
\end{lemma}

Let $(X,\prec)=(\{1,2\},1\prec 2)$. The following Lyndon words will be used often in the sequel.
\begin{align*}
E_{1} &=1, \quad\quad \quad  E_{2}=2, \\
E_{12} &=[E_{1},E_{2}]=12-21,\\
E_{112}&=[E_{1},E_{12}]=[1,[1,2]], \qquad E_{122}=[E_{12},E_{1}]=[[1,2],2],\\
E_{1112}&=[E_{1},E_{112}]=[1,[1,[1,2]]], \qquad E_{1122}=[E_{1},E_{122}]=[1,[[1,2],2]],\\
E_{1222}& =[E_{122},E_{2}]=[[[1,2],2],2].
\end{align*}
For a non-empty set $X$, the free algebra $\K\langle X\rangle$ is the universal enveloping algebra of $\mathfrak{L}(X)$ from \cite[Corollary 5.3.9]{L1997}. Thus $\K\langle X\rangle$ is a free $\K$-module with a PBW basis 
\begin{align}\label{PBW-basis-KX}
\{E_{\alpha_{1}}^{n_{\alpha_{1}}}E_{\alpha_{2}}^{n_{\alpha_{2}}}\cdots E_{\alpha_{m}}^{n_{\alpha_{m}}}|~\alpha_{1}\succ \alpha_{2}\succ \ldots \succ \alpha_{m}, ~\alpha_{i}\in \mathcal{L},~n_{\alpha_{i}}\ge 0,~m\ge 1\}.
\end{align}

\begin{definition}\label{Definition 2.8}
Let $\K\langle x,y\rangle$ be the free algebra generated by $\{x,y\}$. Then
\begin{align*}
(x+y)^n=\sum_{k=0}^{n}\mathcal{SH}_{k,n-k}(y,x),
\end{align*} 
where $\mathcal{SH}_{k,n-k}(y,x)$ is the sum of words $\alpha$ in $(x+y)^{n}$ such that $\alpha(y)=k$ and $\alpha(x)=n-k$, that is, $\mathcal{SH}_{k,n-k}(y,x)=y^{k}\shuffle x^{n-k}$, where $\shuffle$ is the shuffle product. We call it the $(k,n-k)$-\textit{shuffle type polynomials}.
\end{definition}

For instance, $\mathcal{SH}_{0,0}(y,x)=1$, $\mathcal{SH}_{0,1}(y,x)=x$, $\mathcal{SH}_{1,0}(y,x)=y$, $\mathcal{SH}_{1,1}(y,x)=y+x$. In general, $\mathcal{SH}_{i,j}(y,x)$ is recursively defined:
\begin{align}
\label{formula:recursion1-shuffle-poly} \mathcal{SH}_{i,j}(y,x) &= y\mathcal{SH}_{i-1,j}(y,x) + x\mathcal{SH}_{i,j-1}(y,x),\\
\label{formula:recursion2-shuffle-poly} \mathcal{SH}_{i,j}(y,x)&=\mathcal{SH}_{i-1,j}(y,x)y+\mathcal{SH}_{i,j-1}(y,x)x.
\end{align}
Moreover, $[x,\mathcal{SH}_{i,j-1}(y,x)]=[\mathcal{SH}_{i-1,j}(y,x),y]$, and for $k\le i$ and $k\le j$,
\begin{align*}
\mathcal{SH}_{i,j}(y,x)=\sum_{t=0}^{k}\mathcal{SH}_{k-t,t}(y,x)\mathcal{SH}_{i-k+t,j-t}(y,x).
\end{align*}

We mention that the formulas (\ref{formula:recursion1-shuffle-poly}) and (\ref{formula:recursion2-shuffle-poly}) have appeared in \cite[Section 4.2]{R1977}.\\

\subsection{Commutators of Lyndon-Shirshov basis}

\begin{lemma}{\cite[Proposition 5.1.4]{L1997}}\label{lem:alphazeta-is-Lyndon}
Let $\alpha\in \mathcal{L}\setminus X$ and $\operatorname{st}(\alpha)=(\alpha_{L},\alpha_{R})$. Then for any $\zeta\in \mathcal{L}$ such that $\alpha\prec \zeta$, the pair $(\alpha,\zeta)=\operatorname{st}(\alpha\zeta)$ if and only if $\alpha_{R}\succeq \gamma$.
\end{lemma}

The following lemma concerning the commutators of the Lyndon-Shirshov basis is the combination of Lothaire {\cite[Lemma 5.5.3]{L1997}}, Kharchenko {\cite[Lemma 6]{K1999}} and Rosso {\cite[Theorem 1]{R1999}}. 

\begin{lemma}\label{lem:commutators}
Let $(X,\preceq)$ be a totally ordered set, $\mathcal{L}$ the set of Lyndon words in $\langle X \rangle$. For $\alpha,\beta\in \mathcal{L}$, $\alpha\prec \beta$, and $\operatorname{st}(\alpha)=(\alpha_{L},\alpha_{R})$ if $\alpha\in \mathcal{L}\setminus X$, the following hold:\\
$\mathrm{(a)}$ If $\alpha\in X$ or $\alpha_{R}\succeq \beta$, then $[E_{\alpha}, E_{\beta}]=E_{\alpha\beta}$;\\
$\mathrm{(b)}$ If $\alpha_{R} \prec \beta$, then there exits a finite subset $\Gamma \subsetneq \mathcal{L}$ such that $\alpha\beta\in \Gamma$ and
\begin{align}
[E_{\alpha}, E_{\beta}]=\sum_{\gamma\in \Gamma}k_{\gamma}E_{\gamma},
\end{align}
where for every $\gamma\in \Gamma$, $k_{\gamma}\in \Z$, $k_{\alpha\beta}\neq 0$, $\alpha\beta\preceq \gamma\prec \beta$ and $ \gamma(x)=\alpha(x)+\beta(x)~for~all~x\in X$.
\end{lemma}

\begin{proof}
Part (a) follows by Lemma \ref{lem:alphazeta-is-Lyndon}. Part (b) follows from {\cite[Lemma 5.5.3]{L1997}}, {\cite[Lemma 6]{K1999}} and {\cite[Theorem 1]{R1999}}.
\end{proof}

More precisely, we can prove the following result for Part (b) in Lemma \ref{lem:commutators}.

If $\alpha_{R} \prec \beta$, then there exists a finite subset
$\Gamma \subsetneq \mathcal{L}$ such that 
\begin{equation}\label{formula-commutators-LSbasis}
[E_{\alpha}, E_{\beta}]=E_{\alpha\beta}+\sum_{\gamma\in \Gamma}k_{\gamma}E_{\gamma},
\end{equation}
where $k_\gamma\in \mathbb{Z}$,  $\alpha_{R}\preceq (\alpha\beta)_{R}\prec \beta$, and  for every $\gamma\in \Gamma$, $\gamma$ satisfies: $\alpha\beta \prec \gamma \prec \beta$,  $\alpha_{R}\preceq \gamma_{R}\prec \beta $ and $ \gamma(x)=\alpha(x)+\beta(x)$ for all $x\in X$.

For example, let $(X,\preceq)=(\{1,2\},1\prec 2)$. Note that $\st(112)=(1,12)$, $\st(1112)=(1,112)$ and $112\prec 12\prec 2$. Thus 
\begin{align*}
[E_{112},E_{2}] &=[[E_{1},E_{12}],E_{2}]=[[E_{1},E_{2}],E_{12}]+[E_{1},[E_{12},E_{2}]]=[E_{1},E_{122}]=E_{1122},\\
[E_{1112},E_{2}] &=[[E_{1},E_{112}],E_{2}]=[[E_{1},E_{2}],E_{112}]+[E_{1},[E_{112},E_{2}]]\\
&=-[E_{112},E_{12}]+[E_{1},E_{1122}]=E_{11122}-E_{11212},\\
[E_{1122},E_{2}] &=[[E_{1},E_{122}],E_{2}] =[[E_{1},E_{2}],E_{122}]+[E_{1},[E_{122},E_{2}]]\\
&=[E_{12},E_{122}]+[E_{1},E_{1222}]=E_{12122}+E_{11222}.
\end{align*}

The following corollary follows from Lemma \ref{lem:alphazeta-is-Lyndon} and Lemma \ref{lem:commutators}.

\begin{corollary} \label{cor:commutator-LS-basis}
The following hold:
\begin{itemize}
\item [(a)] Let $x\in X$ and $\beta\in \mathcal{L}$. If $x\prec \beta$, then 
\begin{align*}
E_{x}E_{\beta}^{b}=\sum_{c=0}^{b}\binom{b}{c}E_{\beta}^{b-c}E_{x\underbrace{\beta\cdots\beta}_{c}}, \quad b\ge 0.
\end{align*} 
\item [(b)] Let $b_{i}\ge 0$ and $x, \beta_{i}\in \mathcal{L}$ for $1\le i \le n$. If $\beta_{1}\succ \ldots \succ \beta_{k}\succ x\succeq \beta_{k+1}\succ \ldots \succ \beta_{n}$, then  
\begin{align}\label{formula:commutators-x-LSbasis}
E_{x}E_{\beta_{1}}^{b_{1}}\cdots E_{\beta_{n}}^{b_{n}}=\sum_{0\le c_{i}\le b_{i}\atop 1\le i\le k}^{}\binom{b_{1}}{c_{1}}\ldots \binom{b_{k}}{c_{k}}E_{\beta_{1}}^{b_{1}-c_{1}}\cdots E_{\beta_{k}}^{b_{k}-c_{k}}E_{x\underbrace{\beta_{1}\cdots \beta_{1}}_{c_{1}}\cdots \underbrace{\beta_{k}\cdots \beta_{k}}_{c_{k}}}E_{\beta_{k+1}}^{b_{k+1}}\cdots E_{\beta_{n}}^{b_{n}}.
\end{align}
\item [(c)] Let $\alpha,\beta_{1},\beta_{2},\ldots,\beta_{n}\in \mathcal{L}$ and $\beta_{1}\succ \ldots \succ \beta_{k} \succ  \alpha\succeq \beta_{k+1}\succ \ldots \succ \beta_{n}$. Then there exists a finite subset $\Gamma\subsetneq \mathcal{L}$ such that
\begin{align*}
E_{\alpha}E_{\beta_{1}}^{b_{1}} \cdots E_{\beta_{n}}^{b_{n}}= E_{\beta_{1}}^{{b_{1}}}\cdots E_{\beta_{k}}^{b_{k}}E_{\alpha}E_{\beta_{k+1}}^{b_{k+1}}\cdots E_{\beta_{n}}^{b_{n}}+\sum_{\gamma\in \Gamma}^{} k_{\gamma} E_{\beta_{1}}^{c_{1}}\cdots E_{\beta_{\ell}}^{c_{\ell}}E_{\gamma}E_{\beta_{\ell+1}}^{b_{\ell+1}}\cdots E_{\beta_{k+1}}^{b_{k+1}}\cdots E_{\beta_{n}}^{b_{n}},
\end{align*}  
where for every $\gamma$, $k_{\gamma}\in \Z$ and there exists some $\ell$ such that $1\le \ell \le k$, $\beta_{\ell}\succ \gamma\succeq \beta_{\ell+1}$, $0\le c_{s}\le b_{s}$ for $1\le s \le \ell$, $\sum_{s=1}^{\ell}(b_{s}-c_{s})>0$ and $\gamma(x)=\alpha(x)+\sum_{s=1}^{\ell}(b_{s}-c_{s})\beta_{s}(x)$ for $x\in X$.
\item [(d)] Let $\beta,\alpha_{1},\alpha_{2},\ldots,\alpha_{m}\in \mathcal{L}$ and $\alpha_{1}\succ \ldots \succ \alpha_{k} \succeq  \beta \succ\alpha_{k+1}\succ \ldots \succ \alpha_{m}$. 
Then there exists a finite subset $\Gamma\subsetneq \mathcal{L}$ such that 
\begin{align*}
E_{\alpha_{1}}^{a_{1}}E_{\alpha_{2}}^{a_{2}}\cdots E_{\alpha_{m}}^{a_{m}}E_{\beta}= E_{\alpha_{1}}^{{a_{1}}}\cdots E_{\alpha_{k}}^{a_{k}}E_{\beta}E_{\alpha_{k+1}}^{a_{k+1}}\cdots E_{\alpha_{m}}^{a_{m}}+\sum_{\gamma\in \Gamma}^{} k_{\gamma} E_{\alpha_{1}}^{{a_{1}}}\cdots E_{\alpha_{k}}^{a_{k}}\cdots E_{\alpha_{\ell}}^{a_{\ell}}E_{\gamma}E_{\alpha_{\ell+1}}^{c_{\ell+1}}\cdots E_{\alpha_{m}}^{c_{m}},
\end{align*}
where for every $\gamma\in \Gamma$, $k_{\gamma}\in \mathbb{Z}$, there exists some $\ell$ such that $k\le \ell \le m-1$, $\alpha_{\ell}\succeq \gamma\succ \alpha_{\ell+1},~0\le c_{s}\le a_{s}$ for $\ell+1\le s \le m$, $\sum_{s=\ell+1}^{m}(a_{s}-c_{s})>0$ and $\gamma(x)=\beta(x)+\sum_{s=\ell+1}^{m}(a_{s}-c_{s})\alpha_{s}(x)$ for all $x\in X$.
\end{itemize}
\end{corollary}

\section{Noncommutative binomial theorem given by Lyndon-Shirshov basis}\label{Section 2}

In this section, we work with the totally ordered set $X=\{1,2\}$ with the normal order $1\prec 2$.

By Formulas (\ref{PBW-basis-KX}), (\ref{formula:recursion1-shuffle-poly}) and (\ref{formula:commutators-x-LSbasis}), $\mathcal{SH}_{i,j}(E_{2},E_{1})$ can be inductively rewritten as a linear combination of products of the Lyndon-Shirshov basis elements in lexicographic order, that is,
\begin{align*}
\mathcal{SH}_{i,j}(E_{2},E_{1})=\sum_{}^{} C_{E_{\alpha_{1}}^{t_{\alpha_{1}}} \cdots E_{\alpha_{n}}^{t_{\alpha_{n}}}} E_{\alpha_{1}}^{t_{\alpha_{1}}} \cdots E_{\alpha_{n}}^{t_{\alpha_{n}}},
\end{align*}
where $\alpha_{k}\in \mathcal{L},~ t_{\alpha_{k}}\in \mathbb{N}$,
$1\le k\le n, ~\alpha_{1}\succ \ldots\succ \alpha_{n}$, $i=\sum_{k=0}^{n}t_{\alpha_{k}}\alpha_{k}(2)$ and $j=\sum_{k=0}^{n}t_{\alpha_{k}}\alpha_{k}(1)$. 
Now we give the following formula for the coefficient $C_{E_{\alpha_{1}}^{t_{\alpha_{1}}}\cdots E_{\alpha_{n}}^{t_{\alpha_{n}}}}$.

\begin{theorem}\label{thm:noncom-binom-coeff-LSbasis} 
Suppose that $E_{\alpha_{1}}^{t_{\alpha_{1}}}\cdots E_{\alpha_{n}}^{t_{\alpha_{n}}}$ is a term of $\mathcal{SH}_{i,j}(E_{2},E_{1})$, where $\alpha_{k}\in \mathcal{L}_{X}$, $t_{\alpha_{k}}\in \mathbb{N}$, $1\le k\le n $, $\alpha_{1}\succ \ldots\succ \alpha_{n}$, $i=\sum_{k=0}^{n}t_{\alpha_{k}}\alpha_{k}(2)$ and $j=\sum_{k=0}^{n}t_{\alpha_{k}}\alpha_{k}(1)$. Then the coefficient of $E_{\alpha_{1}}^{t_{\alpha_{1}}}E_{\alpha_{2}}^{t_{\alpha_{2}}}\cdots E_{\alpha_{n}}^{t_{\alpha_{n}}}$ in $\mathcal{SH}_{i,j}(E_{2},E_{1})$ is given by
\begin{equation}\label{j-z formula}
C_{E_{\alpha_{1}}^{t_{\alpha_{1}}} \cdots E_{\alpha_{n}}^{t_{\alpha_{n}}}}=\frac{(i+j)!}{\prod_{k=1}^{n}(|\alpha_{k}|!)^{t_{\alpha_{k}}}t_{\alpha_{k}}!}\prod_{k=1}^{n}C_{E_{\alpha_{k}}}^{t_{\alpha_{k}}},
\end{equation}
where $i+j=\sum_{k=1}^{n}|\alpha_{k}|t_{\alpha_{k}}$ and $C_{E_{\alpha_{k}}}$ is the coefficient of the term $E_{\alpha_{k}}$ in $\mathcal{SH}_{\alpha_{k}(2),\alpha_{k}(1)}(E_{2},E_{1})$.
\end{theorem}

\begin{proof}
We prove this result by induction on $i+j$. It holds for $0\le i+j\le 2$, see Example \ref{eg.:SH2} . For every $\alpha_{k}\ne 2$, by Lemma \ref{Lemma 2.4} there exists an unique factorization of $\alpha_{k}$ such that
$$\alpha_{k}=1\underbrace{\alpha_{k,1}\cdots \alpha_{k,1}}_{a_{k,1}} \cdots \underbrace{\alpha_{k,n_{k}}\cdots \alpha_{k,n_{k}}}_{a_{k,n_{k}}},$$
where 
$\alpha_{k,c}\in \mathcal{L}$, $a_{k,c}\in \mathbb{N}$, $1\le c\le n_{k}$ and $\alpha_{k,1}\succ \alpha_{k,2}\succ \ldots\succ \alpha_{k,n_{k}}\succ \alpha_{k}$. 
Thus, by the formulas (\ref{formula:recursion1-shuffle-poly}) and (\ref{formula:commutators-x-LSbasis}),
$E_{\alpha_{k}}$ only occurs in $E_{1}\cdot E_{\alpha_{k,1}}^{a_{k,1}} \cdots E_{\alpha_{k,n_{k}}}^{a_{k,n_{k}}}$ such that 
\begin{align*}
E_{\alpha_{k}}=E_{1\underbrace{\alpha_{k,1}\ldots \alpha_{k,1}}_{a_{k,1}} \ldots \underbrace{\alpha_{k,n_{k}}\ldots \alpha_{k,n_{k}}}_{a_{k,n_{k}}}},
\end{align*}
$E_{\alpha_{k,1}}^{a_{k,1}} \cdots E_{\alpha_{k,n_{k}}}^{a_{k,n_{k}}}$ is a term of $\mathcal{SH}_{\alpha_{k}(2),\alpha_{k}(1)-1}$ and the coeffcient of $E_{\alpha_{k}}$ is
$C_{E_{\alpha_{k,1}}^{a_{k,1}} \cdots E_{\alpha_{k,n_{k}}}^{a_{k,n_{k}}}}$.  Since $\alpha_{k}(2)+\alpha_{k}(1)-1< i+j$,  by induction we have
\begin{equation}\label{coefficient of Lyndon-Shirshov basis}
C_{E_{\alpha_{k}}}=C_{E_{\alpha_{k,1}}^{a_{k,1}} \cdots E_{\alpha_{k,n_{k}}}^{a_{k,n_{k}}}}=\frac{(\sum_{c=1}^{n_{k}}|\alpha_{k,c}|a_{k,c})!}{\prod_{c=1}^{n_{k}}(|\alpha_{k,c}|!)^{a_{k,c}}a_{k,c}!}\prod_{c=1}^{n_{k}}C_{E_{\alpha_{k,c}}}^{a_{k,c}}=\frac{(|\alpha_{k}|-1)!}{\prod_{c=1}^{n_{k}}(|\alpha_{k,c}|!)^{a_{k,c}}a_{k,c}!}\prod_{c=1}^{n_{k}}C_{E_{\alpha_{k,c}}}^{a_{k,c}}.
\end{equation}
For convenience, assume $t_{\alpha_{k}}\ge 1$ for all $1\le k\le n$. Now we consider two cases: $\alpha_{1}=2$ and $\alpha_{1}\ne 2$.
  
\noindent\textit{Case 1}: 
$\alpha_{1}=2$. By Formula (\ref{formula:recursion1-shuffle-poly}), $E_{2}^{t_{2}}E_{\alpha_{2}}^{t_{\alpha_{2}}}\cdots E_{\alpha_{n}}^{t_{\alpha_{n}}}$ appears in $E_{2}\cdot E_{2}^{t_{2}-1}E_{\alpha_{2}}^{t_{\alpha_{2}}}\cdots E_{\alpha_{n}}^{t_{\alpha_{n}}}$, with the following coefficient appearing $\mathcal{SH}_{i-1,j}$:
\begin{align}\label{coefficient-1}
C_{E_{2}^{t_{2}-1}E_{\alpha_{2}}^{t_{\alpha_{2}}}\cdots E_{\alpha_{n}}^{t_{\alpha_{n}}}}.
\end{align}
Similarly, $E_{2}^{t_{2}}E_{\alpha_{2}}^{t_{\alpha_{2}}}\cdots E_{\alpha_{n}}^{t_{\alpha_{n}}}$  appears in 
$E_{1}\cdot E_{2}^{t_{2}}E_{\alpha_{2}}^{t_{\alpha_{2}}}\cdots E_{\alpha_{k-1}}^{t_{\alpha_{k-1}}}(E_{\alpha_{k,1}}^{a_{k,1}}\cdots E_{\alpha_{k,n_{k}}}^{a_{k,n_k}}E_{\alpha_{k}}^{t_{\alpha_{k}}-1})E_{\alpha_{k+1}}^{t_{\alpha_{k+1}}}\cdots E_{\alpha_{n}}^{t_{\alpha_{n}}}$ 
for all $2\le k \le n$ by Formula (\ref{formula:commutators-x-LSbasis}). However, the coefficient of  $E_{2}^{t_{2}}E_{\alpha_{2}}^{t_{\alpha_{2}}}\cdots E_{\alpha_{n}}^{t_{\alpha_{n}}}$ appearing in  the latter can not be determined, because the lexicographic order of the latter may not be right. Without loss of generality, assume that there is some $s$ such that $1\le s\le n_{k}$, $2\le r_{1}<r_{2}<\ldots <r_{s}< k-1$, $\alpha_{k,1}=\alpha_{r_{1}},~\alpha_{k,2}=\alpha_{r_{2}},\ldots,~\alpha_{k,s}=\alpha_{r_{s}}\succ \alpha_{k-1}\succ \alpha_{k,s+1}$ (the proof of the case: $\alpha_{k-1}=\alpha_{k,s+1}$ is similar to the case: $\alpha_{k-1}\succ \alpha_{k,s+1}$). 
Now we have $2\succ \alpha_{2}\succ \ldots \succ \alpha_{r_{1}}\succ \ldots \succ \alpha_{r_{2}} \succ \ldots \succ \alpha_{r_{s}} \succ\ldots\succ \alpha_{k-1}\succ \alpha_{k,s+1}\succ\ldots\succ  \alpha_{k,n_{k}}\succ \alpha_{k} \succ \alpha_{k+1}\succ \ldots \succ \alpha_{n}$.
It follows from Lemma \ref{Lemma 2.4} and Formula (\ref{formula:commutators-x-LSbasis}) that $E_{2}^{t_{2}}E_{\alpha_{2}}^{t_{\alpha_{2}}}\cdots E_{\alpha_{n}}^{t_{\alpha_{n}}}$ appears in 
\begin{align*}
E_{1}\cdot E_{2}^{t_{2}}E_{\alpha_{2}}^{t_{\alpha_{2}}}\cdots E_{\alpha_{r_{1}}}^{t_{\alpha_{r_{1}}}+a_{k,1}}\cdots E_{\alpha_{r_{2}}}^{t_{\alpha_{r_{2}}}+a_{k,2}}\cdots E_{\alpha_{r_{s}}}^{t_{\alpha_{r_{s}}}+a_{k,s}}\cdots E_{\alpha_{k-1}}^{t_{\alpha_{k-1}}}E_{\alpha_{k,s+1}}^{a_{k,s+1}}\cdots E_{\alpha_{k,n_{k}}}^{a_{k,n_{k}}} E_{\alpha_{k}}^{t_{\alpha_{k}}-1}E_{\alpha_{k+1}}^{t_{\alpha_{k+1}}}\cdots E_{\alpha_{n}}^{t_{\alpha_{n}}},
\end{align*} 
with the coefficient appearing in  $\mathcal{SH}_{i,j-1}$:
\begin{equation}\label{formula 9}
\begin{array}{l}
\qquad \begin{aligned}
& \binom{t_{\alpha_{r_{1}}}+a_{k,1}}{a_{k,1}}\binom{t_{\alpha_{r_{2}}}+a_{k,2}}{a_{k,2}}\ldots \binom{t_{\alpha_{r_{s}}}+a_{k,s}}{a_{k,s}} \cdot \\
& C_{ E_{2}^{t_{2}}E_{\alpha_{2}}^{t_{\alpha_{2}}}\cdots E_{\alpha_{r_{1}}}^{t_{\alpha_{r_{1}}}+a_{k,1}}\cdots E_{\alpha_{r_{2}}}^{t_{\alpha_{r_{2}}}+a_{k,2}}\cdots E_{\alpha_{r_{s}}}^{t_{\alpha_{r_{s}}}+a_{k,s}}\cdots E_{\alpha_{k-1}}^{t_{\alpha_{k-1}}}E_{\alpha_{k,s+1}}^{a_{k,s+1}}\cdots E_{\alpha_{k,n_{k}}}^{a_{k,n_{k}}} E_{\alpha_{k}}^{t_{\alpha_{k}}-1}E_{\alpha_{k+1}}^{t_{\alpha_{k+1}}}\cdots E_{\alpha_{n}}^{t_{\alpha_{n}}}}. \\
\end{aligned}
\end{array}
\end{equation}
Since $|\alpha_{1}|=|2|=1$, by induction on $i+j$ the formula (\ref{coefficient-1}) yields:
\begin{align*}
\frac{(\sum_{\ell=1}^{n}|\alpha_{\ell}|t_{\alpha_{\ell}}-1)!|\alpha_{1}|t_{\alpha_{1}}}{\prod_{\ell=1}^{n}(|\alpha_{\ell}|!)^{t_{\alpha_{\ell}}}t_{\alpha_{\ell}}!}\left(\prod_{\ell=1}^{n}C_{E_{\alpha_{\ell}}}^{t_{\alpha_{\ell}}}\right).
\end{align*}
Using the relation $\alpha_{k,u}=\alpha_{r_{u}}$ for $1\le u\le s$, Formulas (\ref{formula:commutators-x-LSbasis}) and (\ref{coefficient of Lyndon-Shirshov basis}), we conclude inductively that for $2\le k \le n$, the formula (\ref{formula 9}) amounts to
\begin{align*}
& \frac{\left(t_{2}+\sum_{\ell=2,\ell\neq k}^{n}|\alpha_{\ell}|t_{\alpha_{\ell}}+\sum_{c=1}^{n_{k}}|\alpha_{k,c}|a_{k,c}+|\alpha_{k}|(t_{\alpha_{k}}-1)\right)!C_{E_{2}}^{t_{2}}\left(\prod_{c=1}^{n_{k}}C_{E_{\alpha_{k,c}}}^{a_{k,c}}\right)C_{E_{\alpha_{k}}}^{t_{\alpha_{k}}-1}\left(\prod_{\ell=2,\ell\neq k}^{n}C_{E_{\alpha_{\ell}}}^{t_{\alpha_{\ell}}}\right)}{t_{2}!(\prod_{\ell=2,\ell\neq k}^{n}(|\alpha_{\ell}|!)^{t_{\alpha_{\ell}}}t_{\alpha_{\ell}}!)(\prod_{c=1}^{n_{k}}(|\alpha_{k,c}|!)^{a_{k,c}}a_{k,c}!)(|\alpha_{k}|!)^{t_{\alpha_{k}}-1}(t_{\alpha_{k}}-1)!}\\
=&\frac{(\sum_{\ell=1}^{n}|\alpha_{\ell}|t_{\alpha_{\ell}}-1)!}{(\prod_{\ell=1,\ell\neq k}^{n}(|\alpha_{\ell}|!)^{t_{\alpha_{\ell}}}t_{\alpha_{\ell}}!)}\left(\prod_{\ell=1,\ell\neq k}^{n}C_{E_{\alpha_{\ell}}}^{t_{\alpha_{\ell}}}\right)C_{E_{\alpha_{k}}}^{t_{\alpha_{k}}-1}
\frac{|\alpha_{k}|t_{\alpha_{k}}(|\alpha_{k}|-1)!}{(|\alpha_{k}|!)^{t_{\alpha_{k}}}t_{\alpha_{k}}!\prod_{c=1}^{n_{k}}(|\alpha_{k,c}|!)^{a_{k,c}}a_{k,c}!}\left(\prod_{c=1}^{n_{k}}C_{E_{\alpha_{k,c}}}^{a_{k,c}}\right)\\
=&\frac{(\sum_{\ell=1}^{n}|\alpha_{\ell}|t_{\alpha_{\ell}}-1)!}{(\prod_{\ell=1,\ell\neq k}^{n}(|\alpha_{\ell}|!)^{t_{\alpha_{\ell}}}t_{\alpha_{\ell}}!)}\left(\prod_{\ell=1,\ell\neq k}^{n}C_{E_{\alpha_{\ell}}}^{t_{\alpha_{\ell}}}\right)\frac{|\alpha_{k}|t_{\alpha_{k}}}{(|\alpha_{k}|!)^{t_{\alpha_{k}}}t_{\alpha_{k}}!}C_{E_{\alpha_{k}}}^{t_{\alpha_{k}}}\\
=& \frac{(\sum_{\ell=1}^{n}|\alpha_{\ell}|t_{\alpha_{\ell}}-1)!|\alpha_{k}|t_{\alpha_{k}}}{\prod_{\ell=1}^{n}(|\alpha_{\ell}|!)^{t_{\alpha_{\ell}}}t_{\alpha_{\ell}}!}\left(\prod_{\ell=1}^{n}C_{E_{\alpha_{\ell}}}^{t_{\alpha_{\ell}}}\right).
\end{align*}
Finally, we sum all the coefficients appearing in  $\mathcal{SH}_{i-1,j}$ and $\mathcal{SH}_{i,j-1}$ for all $1\le k\le n$:
\begin{align*}
C_{E_{\alpha_{1}}^{t_{\alpha_{1}}} \cdots E_{\alpha_{n}}^{t_{\alpha_{n}}}}
& =\sum_{k=1}^{n}\frac{(\sum_{\ell=1}^{n}|\alpha_{\ell}|t_{\alpha_{\ell}}-1)!|\alpha_{k}|t_{\alpha_{k}}}{(\prod_{\ell=1}^{n}(|\alpha_{\ell}|!)^{t_{\alpha_{\ell}}}t_{\alpha_{\ell}}!)}\left(\prod_{\ell=1}^{n}C_{E_{\alpha_{\ell}}}^{t_{\alpha_{\ell}}}\right)\\
& =\frac{(\sum_{\ell=1}^{n}|\alpha_{\ell}|t_{\alpha_{\ell}})!}{\prod_{\ell=1}^{n}(|\alpha_{\ell}|!)^{t_{\alpha_{\ell}}}t_{\alpha_{\ell}}!}\prod_{\ell=1}^{n}C_{E_{\alpha_{\ell}}}^{t_{\alpha_{\ell}}}.\\
& =\frac{(i+j)!}{\prod_{k=1}^{n}(|\alpha_{k}|!)^{t_{\alpha_{k}}}t_{\alpha_{k}}!}\prod_{k=1}^{n}C_{E_{\alpha_{k}}}^{t_{\alpha_{k}}}.
\end{align*}

\noindent\textit{Case 2}: $\alpha_{1}\neq 2$. From Formula (\ref{formula:recursion1-shuffle-poly}), we can see that $E_{\alpha_{1}}^{t_{\alpha_{1}}} \cdots E_{\alpha_{n}}^{t_{\alpha_{n}}}$ only occurs in $E_{1}\cdot \mathcal{SH}_{i,j-1}$. The rest is similar to the proof of Case 1.
\end{proof}

Note that Formula (\ref{coefficient of Lyndon-Shirshov basis}) gives the coefficients of the Lyndon-Shirshov basis elements $E_{\alpha}$
in $\mathcal{SH}_{\alpha(2),\alpha(1)}$, as a supplement to Formula (\ref{j-z formula}).

 Now we have the following normalized  forms of  the shuffle type  polynomials  and the binomial expansions from Theorem \ref{thm:noncom-binom-coeff-LSbasis}.

\begin{corollary}\label{cor:noncomm-binom-thm-LSbasis}
Let $i,j,d\ge 0$. The following hold:
\begin{align}
\label{j-z formula 2}
&\mathcal{SH}_{i,j}=\sum_{\sum_{k=1}^{n}t_{\alpha_{k}}\alpha_{k}(2)=i \atop \sum_{k=1}^{n}t_{\alpha_{k}}\alpha_{k}(1)=j}^{}\frac{(i+j)!}{t_{\alpha_{1}}! \cdots t_{\alpha_{n}}!}\left(\frac{C_{E_{\alpha_{1}}}E_{\alpha_{1}}}{|\alpha_{1}|!}\right)^{t_{\alpha_{1}}} \cdots \left(\frac{C_{E_{\alpha_{n}}}E_{\alpha_{n}}}{|\alpha_{n}|!}\right)^{t_{\alpha_{n}}},\\
\label{j-z formula 3}
& (E_{1}+E_{2})^{d}=\sum_{\sum_{k=1}^{n}t_{\alpha_{k}}|\alpha_{k}|=d}^{}\frac{d!}{t_{\alpha_{1}}! \cdots t_{\alpha_{n}}!}\left(\frac{C_{E_{\alpha_{1}}}E_{\alpha_{1}}}{|\alpha_{1}|!}\right)^{t_{\alpha_{1}}} \cdots \left(\frac{C_{E_{\alpha_{n}}}E_{\alpha_{n}}}{|\alpha_{n}|!}\right)^{t_{\alpha_{n}}},
\end{align}
where $\alpha_{k}\in \mathcal{L}_{X}$, $t_{\alpha_{k}}\in \mathbb{N}$, $C_{E_{\alpha_{k}}}$ is determined by Formulas (\ref{coefficient of Lyndon-Shirshov basis}), (\ref{j-z formula}),  $\alpha_{1}\succ \ldots\succ \alpha_{n}$ and $1\le k\le n$.
\end{corollary}
The normalized forms of Formulas (\ref{j-z formula 2}), (\ref{j-z formula 3}) are similar to the classical Bell polynomials introduced by Bell \cite{B1927,B1934} (see Formula (\ref{classical Bell polynomials})). We will discuss their relationship in Section \ref{Section 4}.

\begin{corollary}\label{cor:particular-case-noncom-binom}
The following hold: 
\begin{itemize}
\item [(a)] For all $s,t\ge 0$ and $s+t\ge 1$,
\begin{align*}
C{_{E_{\scriptsize \underbrace{1\cdots1}_{s}\underbrace{2\cdots2}_{t}}}}=1.
\end{align*}
\item [(b)] If $\Char\K=p$, then
\begin{align*}
\mathcal{SH}_{k,p-k}(E_{2},E_{1})=\sum_{\alpha(2)=k\atop \alpha(1)=p-k}^{}C_{E_{\alpha}}E_{\alpha},\quad 1\le k\le p-1.
\end{align*} 
Therefore,
\begin{align*}
(E_{1}+E_{2})^{p}=E_{2}^{p}+\sum_{|\alpha|=p}^{}C_{E_{\alpha}}E_{\alpha}+E_{1}^{p}.
\end{align*}

\item [(c)] Let $\Char\K=p$. In particular, if $E_{\alpha}=0$ for all $\alpha$ such that $\alpha(1)\ge 2$ and $|\alpha|=p$, then 
\begin{equation*}
\mathcal{SH}_{k,p-k}(E_{2},E_{1})=\left\{
\begin{aligned}
&~ E_{1\scriptsize \underbrace{2\cdots2}_{p-1} } , & k=p-1, \\
&~ 0 , & 1\le k\le p-2.
\end{aligned}
\right.
\end{equation*}
Therefore,
\begin{align*}
(E_{1}+E_{2})^{p}=E_{2}^{p}+E_{1\scriptsize\underbrace{2\cdots2}_{p-1}}+E_{1}^{p}.
\end{align*}
Similarly, if $E_{\alpha}=0$ for all $\alpha$ such that $\alpha(2)\ge 2$ and $|\alpha|=p$, then $\mathcal{SH}_{k,p-k}(E_{2},E_{1})=0$ for all $2\le k\le p-1$ and $\mathcal{SH}_{1,p-1}(E_{2},E_{1})=E_{\scriptsize  \underbrace{1\cdots1}_{p-1}2}$. Therefore,
$(E_{1}+E_{2})^{p}=E_{2}^{p}+E_{\scriptsize\underbrace{1\cdots1}_{p-1}2}+E_{1}^{p}$.
\end{itemize}
\end{corollary}

\begin{proof}
By Formulas (\ref{coefficient of Lyndon-Shirshov basis}), (\ref{j-z formula}) and $C_{E_{2}}=1$, we have 
\begin{align*}
C{_{E_{\tiny\underbrace{1\cdots1}_{s}\underbrace{2\cdots2}_{t}}}}=
C{_{E_{\tiny \underbrace{1\cdots1}_{s-1}\underbrace{2\cdots2}_{t}}}}=\ldots=C{_{E_{1 \tiny \underbrace{2\cdots2}_{t}}}}=C_{E_{2}^{t}}=(C_{E_{2}})^{t}=1.
\end{align*}

If $\Char\K=p$, then we see that the coefficients of terms in $\mathcal{SH}_{k,p-k}$ ($1\le k\le p-1$) are zero, except for the terms of length $p$, by Formulas (\ref{j-z formula}) and (\ref{coefficient of Lyndon-Shirshov basis}). Thus Part (b) holds. Part (c) follows from Parts (a) and (b). 
\end{proof}

In general,  the formulas (\ref{j-z formula}), (\ref{j-z formula 2}) and (\ref{j-z formula 3}) stay the same for the multinomial expansions. Let $(X,\preceq)=(\{1,\ldots,m\},1\prec \ldots \prec m)$ and $i_{x}\in \mathbb{N}$ for $x\in X$. The $(i_{m},\ldots,i_{1})$-shuffle type polynomial $\mathcal{SH}_{i_{m},\ldots,i_{1}}:=\mathcal{SH}_{i_{m},\ldots,i_{1}}(E_{m},\ldots,E_{1})$ is defined by:
\begin{align*}
(E_{1}+\ldots+E_{m})^{d}=\sum_{i_{1}+\ldots+i_{m}=d}^{}\mathcal{SH}_{i_{m},\ldots,i_{1}}(E_{m},\ldots,E_{1}),
\end{align*}
where $\mathcal{SH}_{i_{m},\ldots,i_{1}}(E_{m},\ldots,E_{1})$ is the sum of the words $\alpha$ in $(E_{1}+\ldots+E_{m})^{d}$ such that $\alpha(E_{x})=\alpha(x)=i_{x}$ for all $x\in X$.

Furthermore, $\mathcal{SH}_{i_{m},\ldots,i_{1}}(E_{m},\ldots,E_{1})=E_{m}^{i_{m}}\shuffle \ldots \shuffle E_{1}^{i_{1}}$, where $\shuffle$ is the shuffle product. Note that $\mathcal{SH}_{i_{m},\ldots,i_{1}}$ has the recursion:
\begin{align}\label{formula:recursion-SH_im}
\mathcal{SH}_{i_{m},\ldots,i_{1}}=E_{m}\mathcal{SH}_{i_{m}-1,i_{m-1},\ldots ,i_{1}}+\ldots+E_{1}\mathcal{SH}_{i_{m},i_{m-1},\ldots ,i_{1}-1}.
\end{align}
From Formulas (\ref{PBW-basis-KX}), \ref{formula:recursion-SH_im} and (\ref{formula:commutators-x-LSbasis}), $\mathcal{SH}_{i_{m},\ldots,i_{1}}$ can be inductively rewritten as a linear combination of the form $E_{\alpha_{1}}^{t_{\alpha_{1}}}\ldots E_{\alpha_{n}}^{t_{\alpha_{n}}}$
such that $\alpha_{k}\in \mathcal{L}_{X}$, $t_{\alpha_{k}}\in \mathbb{N}$, $1\le k\le n$, $\alpha_{1}\succ \ldots\succ \alpha_{n}$ and
$i_{x}=\sum_{k=0}^{n}t_{\alpha_{k}}\alpha_{k}(x)$ for all $x\in X$. We still denote by $C_{E_{\alpha_{1}}^{t_{\alpha_{1}}}E_{\alpha_{2}}^{t_{\alpha_{2}}}\ldots E_{\alpha_{n}}^{t_{\alpha_{n}}}}$ the coefficient of the term $E_{\alpha_{1}}^{t_{\alpha_{1}}} \ldots E_{\alpha_{n}}^{t_{\alpha_{n}}}$ in $\mathcal{SH}_{i_{m},\ldots,i_{1}}$, and give a similar formula for it.

\begin{corollary}\label{cor:noncom-multinomial-thm}
Let $(X,\preceq)=(\{1,\ldots,m\},1\prec \ldots \prec m)$. Suppose that $E_{\alpha_{1}}^{t_{\alpha_{1}}}\cdots E_{\alpha_{n}}^{t_{\alpha_{n}}}$ is a term of $\mathcal{SH}_{i_{m},\ldots,i_{1}}$, where $\alpha_{k}\in \mathcal{L}_{X}$, $t_{\alpha_{k}}\in \mathbb{N}$, $1\le k\le n$, $\alpha_{1}\succ \ldots \succ \alpha_{n}$ and $i_{x}=\sum_{k=0}^{n}t_{\alpha_{k}}\alpha_{k}(x)$ for all $x\in X$. 
Then the coefficient of $E_{\alpha_{1}}^{t_{\alpha_{1}}} \cdots E_{\alpha_{n}}^{t_{\alpha_{n}}}$ in $\mathcal{SH}_{i_{m},\ldots,i_{1}}$ is given by
\begin{align*}
C_{E_{\alpha_{1}}^{t_{\alpha_{1}}} \cdots E_{\alpha_{n}}^{t_{\alpha_{n}}}}=\frac{(\sum_{x\in X}i_{x})!}{\prod_{k=1}^{n}\left(|\alpha_{k}|!\right)^{t_{\alpha_{k}}}t_{\alpha_{k}}!}\prod_{k=1}^{n}C_{E_{\alpha_{k}}}^{t_{\alpha_{k}}},
\end{align*} 
where $\sum_{x\in X}i_{x}=\sum_{k=1}^{n}|\alpha_{k}|t_{\alpha_{k}}$ and $C_{E_{\alpha_{k}}}$ is the coefficient of the term $E_{\alpha_{k}}$ in
 $\mathcal{SH}_{\alpha_{k}(m),\ldots,\alpha_{k}(1)}$. 
\end{corollary}

\begin{proof}
By Lemma \ref{Lemma 2.4}, Formulas (\ref{formula:recursion-SH_im}), (\ref{formula:commutators-x-LSbasis}) and the induction on $\sum_{x\in X}i_{x}$, we conclude that $C_{E_{\alpha_{k}}}$ is also given by Formula (\ref{coefficient of Lyndon-Shirshov basis}). Thus, the rest of the proof is the same as that in Theorem \ref{thm:noncom-binom-coeff-LSbasis}.
\end{proof}

Note that Corollary \ref{cor:noncomm-binom-thm-LSbasis} and Corollary \ref{cor:particular-case-noncom-binom} can be also generalized to the multinomial case.\\

We compute the following shuffle type polynomials $\mathcal{SH}_{k,n-k}(y,x)$ in terms of the Lyndon-Shirshov basis for $2\le n\le 4$, which will be applied later on.  Let $(X,\prec)=(\{1,2\},1\prec 2)$. For the sake of the notations of commutators,  we write  $E_{1}$ for $x$ and $E_{2}$ for  $y$.
\begin{example}\label{eg.:SH2} 
If $n=2$, then 
\begin{align*} 
\mathcal{SH}_{0,2}(E_2,E_1) &=E_{1}^{2},\\
\mathcal{SH}_{1,1}(E_2,E_1) &=E_{2}E_{1}+E_{1}E_{2}=2E_2E_1+E_{12},\\
\mathcal{SH}_{2,0}(E_2,E_1) &=E_{2}^{2}.
\end{align*}
Thus $(E_1+E_2)^{2}=(E_{1}+E_{2})^{2}=E_{1}^{2}+2E_2E_1+E_{12}+E_{2}^{2}$.\\

For multinomial case, let $(X,\prec)=(\{1,2,3\},1\prec 2\prec 3)$ in Corollary \ref{cor:noncom-multinomial-thm}.  If $n=2$, then
$$(E_{1}+E_{2}+E_{3})^{2}=E_{1}^{2}+2E_{2}E_{1}+E_{12}+E_{2}^{2}+2E_{3}E_{1}+E_{13}+2E_{3}E_{2}+E_{23}+E_{3}^{2}.$$
\end{example}

\begin{example}\label{eg.:SH3} 
For $n=3$, we have
\begin{align*} 
\mathcal{SH}_{0,3}(E_{2},E_{1}) &=E_{1}^{3},\\
\mathcal{SH}_{1,2}(E_{2},E_{1}) &=E_{2}E_{1}^{2}+E_{1}E_{2}E_{1}+E_{1}^{2}E_{2}=3E_2E_1^{2}+3E_{12}E_1+E_{112},\\
\mathcal{SH}_{2,1}(E_{2},E_{1}) &=E_{2}^{2}E_{1}+E_{2}E_{1}E_{2}+E_{1}E_{2}^{2}=3E_2^{2}E_1+3E_2E_{12}+E_{122},\\
\mathcal{SH}_{3,0}(E_{2},E_{1}) &=E_{2}^{3}.
\end{align*}
Therefore, 
\begin{align*}
(E_{1}+E_{2})^{3}= E_{1}^{3}+ 3E_2E_1^{2}+3E_{12}E_1+E_{112} + 3E_2^{2}E_1+3E_2E_{12}+E_{122} + E_{2}^{3}.
\end{align*}
\end{example}

\begin{example}\label{eg.:SH4} 
For $n=4$, we obtain
\begin{align*}
\mathcal{SH}_{0,4}(E_{2},E_{1}) &=E_{1}^{4},\\
\mathcal{SH}_{1,3}(E_{2},E_{1}) &=4E_{2}E_{1}^{3}+6E_{12}E_{1}^{2}+4E_{112}E_{1}+E_{1112},\\
\mathcal{SH}_{2,2}(E_{2},E_{1}) &=6E_{2}^{2}E_{1}^{2}+12E_{2}E_{12}E_{1}+4E_{2}E_{112}+4E_{122}E_{1}+3E_{12}^{2}+E_{1122},\\
\mathcal{SH}_{3,1}(E_{2},E_{1}) &=4E_{2}^{3}E_{1}+6E_{2}^{2}E_{12}+4E_{2}E_{122}+E_{1222},\\
\mathcal{SH}_{4,0}(E_{2},E_{1}) &=E_{2}^{4},
\end{align*}
Therefore, 
\begin{align*}
(E_{1}+E_{2})^{4}= & E_{1}^{4}+ 4E_{2}E_{1}^{3}+6E_{12}E_{1}^{2}+4E_{112}E_{1}+E_{1112} \\
&+ 6E_{2}^{2}E_{1}^{2}+12E_{2}E_{12}E_{1}+4E_{2}E_{112}+4E_{122}E_{1}+3E_{12}^{2}+E_{1122}\\
& + 4E_{2}^{3}E_{1}+6E_{2}^{2}E_{12}+4E_{2}E_{122}+E_{1222}
+ E_{2}^{4},
\end{align*}
\end{example}
For  examples  of higher orders we  refer to Appendix \ref{Appendix A}.

\section{Noncommutative binomial theorem given by the shuffle type polynomials with respect to a $\sigma$-derivation}\label{Section 3} 

\subsection{Binomial theorem given by the shuffle type polynomials with respect to a $\sigma$-derivation}\label{Section 3.1} 
\qquad

Let $A$ be a unital associative algebra and $x\in A$. Suppose that there is an algebra homomorphism $\sigma: A\rightarrow A$. We define a left $\sigma$-adjoint action $\operatorname{ad}_{\sigma}x: A\rightarrow A$  by
\begin{align*}
\operatorname{ad}_{\sigma}x(w):=[x,w]_{\sigma}=xw-\sigma(w)x, \quad  w\in A
\end{align*}
It is easy to see that 
\begin{align}
\operatorname{ad}_{\sigma}x(uv)=\operatorname{ad}_{\sigma}x(u)v+\sigma(u)\operatorname{ad}_{\sigma}x(v),\quad  u,v\in A.
\end{align}
Thus $\operatorname{ad}_{\sigma}x$ is a $\sigma$-derivaion. 

 Let $x,y\in A$ and $\sigma$ as above.  We define 
  $$(\operatorname{ad}_{\sigma}x+y)(z):=\operatorname{ad}_{\sigma}x(z)+yz, \qquad \forall z\in A.$$ 
It is clear that $(\operatorname{ad}_{\sigma}x+y)(1)=y$ and $(\operatorname{ad}_{\sigma}x+y)^{2}(1)=\operatorname{ad}_{\sigma}x(y)+y^{2}=xy-\sigma(y)x+y^{2}$.  

Using the commuting relation $xy=\sigma(y)x + \operatorname{ad}_{\sigma}x(y)$, we obtain the following examples of the binomial formulas of low orders.
\begin{example}
\begin{align*}
(x+y)^{2} 
= & x^{2}+(y+\sigma(y))x+(\operatorname{ad}_{\sigma}x(y)+y^{2})\\
= & x^{2}+(\left(\operatorname{ad}_{\sigma}x+y)\circ \sigma+\sigma\circ (\operatorname{ad}_{\sigma}x+y)\right)(1)x+(\operatorname{ad}_{\sigma}x+y)^{2}(1)\\
= & \mathcal{SH}_{0,2}(\operatorname{ad}_{\sigma}x+y,\sigma)(1) x^{2} +  \mathcal{SH}_{1,1}(\operatorname{ad}_{\sigma}x+y,\sigma)(1) x + \mathcal{SH}_{2,0}(\operatorname{ad}_{\sigma}x+y,\sigma)(1).\\
(x+y)^{3} =& x^{3}+(y+\sigma(y)+\sigma^{2}(y))x^{2}+(\operatorname{ad}_{\sigma}x(y)+\operatorname{ad}_{\sigma}x(\sigma(y))+y^{2}+y\sigma(y)+\sigma(\operatorname{ad}_{\sigma}x(y))\\
&+ \sigma(y)\sigma(y))x+(\operatorname{ad}_{\sigma}^{2}x(y)+\sigma(y)\operatorname{ad}_{\sigma}x(y)+\operatorname{ad}_{\sigma}x(y)y+y\operatorname{ad}_{\sigma}x(y)+y^{3})\\
 = &  x^{3}+((\operatorname{ad}_{\sigma}x+y)\circ \sigma^{2}+\sigma\circ (\operatorname{ad}_{\sigma}x+y)\circ \sigma+\sigma^{2}\circ (\operatorname{ad}_{\sigma}x+y))(1)x^{2}+((\operatorname{ad}_{\sigma}x+y)^{2}\circ \sigma\\
& + (\operatorname{ad}_{\sigma}x+y)\circ \sigma\circ (\operatorname{ad}_{\sigma}x+y)+\sigma\circ (\operatorname{ad}_{\sigma}x+y)^{2})(1)x+(\operatorname{ad}_{\sigma}x+y)^{3}(1)\\
= & \mathcal{SH}_{0,3}(\operatorname{ad}_{\sigma}x+y,\sigma)(1) x^{3} +
\mathcal{SH}_{1,2}(\operatorname{ad}_{\sigma}x+y,\sigma)(1) x^{2} 
+ \mathcal{SH}_{2,1}(\operatorname{ad}_{\sigma}x+y,\sigma)(1) x \\
& + \mathcal{SH}_{3,0}(\operatorname{ad}_{\sigma}x+y,\sigma)(1).
\end{align*}
\end{example}

In general, we have the following noncommutative binomial formula.

\begin{theorem}\label{thm:binom-formula-shuffle-poly-sigma-derivation}
Let $A$ be a unital associative algebra,  $x,y\in A$, and $\sigma$ an algebra endomorphism of $A$. 
Then the following holds for $n\ge 0$,
\begin{align}\label{binomial-shuffle-sigma}{}
(x+y)^{n}=\sum_{k=0}^{n}\widehat{\mathcal{SH}}_{k,n-k}(1)x^{n-k},
\end{align}
where $\widehat{\mathcal{SH}}_{k,n-k}:=\mathcal{SH}_{k,n-k}(\operatorname{ad}_{\sigma}x+y,\sigma)$ and $\operatorname{ad}_{\sigma}x(a):=xa-\sigma(a)x$ for $a\in A$.
\end{theorem}

\begin{proof}
By Formula (\ref{formula:recursion1-shuffle-poly}) we have
\begin{align}\label{formula:noncommutative-binomial-2}
\widehat{\mathcal{SH}}_{i,j}=\sigma\circ \widehat{\mathcal{SH}}_{i,j-1}+(\operatorname{ad}_{\sigma}x+y)\circ \widehat{\mathcal{SH}}_{i-1,j}.
\end{align}
By induction on $n$, we obtain
\begin{align*}
&(x+y)^{n+1}\\
 =&(x+y)\sum_{k=0}^{n}\widehat{\mathcal{SH}}_{k,n-k}(1)x^{n-k}\\
=&\sum_{k=0}^{n}\left(\sigma(\widehat{\mathcal{SH}}_{k,n-k}(1))x+\operatorname{ad}_{\sigma}x(\widehat{\mathcal{SH}}_{k,n-k}(1))+y\widehat{\mathcal{SH}}_{k,n-k}(1)\right)x^{n-k}\\
=& \sigma^{n+1}(1)x^{n+1}+\sum_{k=1}^{n}\sigma \left(\widehat{\mathcal{SH}}_{k,n-k}(1)\right)x^{n-k+1}+\sum_{k=0}^{n-1}(\operatorname{ad}_{\sigma}x+y)\left(\widehat{\mathcal{SH}}_{k,n-k}(1)\right)x^{n-k}+(\operatorname{ad}_{\sigma}x+y)^{n+1}(1)\\
=& \sigma^{n+1}(1)x^{n+1}+\sum_{k=1}^{n}\left(\sigma\circ \widehat{\mathcal{SH}}_{k,n-k}+(\operatorname{ad}_{\sigma}x+y)\circ \widehat{\mathcal{SH}}_{k-1,n-k+1}\right)(1)x^{n-k+1}+(\operatorname{ad}_{\sigma}x+y)^{n+1}(1)\\
=&\sum_{k=0}^{n+1}\widehat{\mathcal{SH}}_{k,n+1-k}(1)x^{n+1-k}.
\end{align*}
\end{proof}

For example, $\widehat{\mathcal{SH}}_{0,0}=\id_{A}$, $\widehat{\mathcal{SH}}_{0,1}=\sigma$ and  $\widehat{\mathcal{SH}}_{1,0}=\operatorname{ad}_{\sigma} x +y$. Then $\widehat{\mathcal{SH}}_{0,0}(1)=1$, $\widehat{\mathcal{SH}}_{0,1}(1)=\sigma(1)=1$ and $\widehat{\mathcal{SH}}_{1,0}=(\operatorname{ad}_{\sigma} x +y)(1)=y$. In general, it is easy to see  that $\widehat{\mathcal{SH}}_{0,n}(1)=\sigma^{n}(1)=1$ and 
\begin{align*}
\widehat{\mathcal{SH}}_{n,0}(1)=(\operatorname{ad}_{\sigma}x+y)^{n}(1).
\end{align*}
Theorem \ref{thm:noncom-binom-coeff-LSbasis} gives an explicit formula for $\mathcal{SH}_{k,n-k}(\operatorname{ad}_{\sigma}x+y,\sigma)$ by the Lyndon-Shirshov basis.

Formula (\ref{binomial-shuffle-sigma}) may provides a rough solution to the question of Mansour--Schork \cite[Research problem 8.1]{MS2016}, which requires a binomial formula for the Ore extension $\mathbb{C}[y][x;\sigma,\delta]$. Note that $xy=\sigma(y)x+\delta(y)$. Using Formula (\ref{binomial-shuffle-sigma}), we have
\begin{align*}
(x+y)^{n}=\sum_{k=0}^{n}\mathcal{SH}_{k,n-k}\left(\delta+y,\sigma\right)(1)x^{n-k}.
\end{align*}
The case $\sigma=id$ has an explicit expression given by the classical Bell polynomials, see Formula (\ref{answer-for-case-sigma=id}).

\begin{remark}\label{rmk:SHn,n-k(1)-to-Bell-diff-poly}
It is clear that the following holds:
\begin{align}\label{formula:sigma-adjoint-derivation}
\sigma^{m}\circ (\operatorname{ad}_{\sigma}x+y)=\left(\operatorname{ad}_{\sigma}(\sigma^{m}(x))+\sigma^{m}(y)\right)\circ \sigma^{m},\qquad  m\ge 0.
\end{align}
Denote by $D_m$ the operator 
$\operatorname{ad}_{\sigma}(\sigma^{m}(x))+\sigma^{m}(y),  m\ge 0$. 
Note that $\widehat{\mathcal{SH}}_{0,n}=\sigma^{n}$. By Formulas (\ref{formula:noncommutative-binomial-2}) and (\ref{formula:sigma-adjoint-derivation}) we can prove that 
\begin{align}\label{shuffle-sigma}
\widehat{\mathcal{SH}}_{k,n-k}=\sum_{0\le m_{1}\le m_{2}\le \ldots \le m_{k}\le  n-k }^{}D_{m_{1}}D_{m_{2}}\ldots D_{m_{k}}\sigma^{n-k}, \quad  1\le k\le n.
\end{align}
If $\sigma$ is an automorphism of $A$, then Formula (\ref{formula:sigma-adjoint-derivation}) implies 
\begin{align}\label{formula 20}
D_{m}=\sigma^{m}\circ D_{0}\circ\sigma^{-m}.
\end{align}
Since $\sigma(1)=1$, by Formulas (\ref{shuffle-sigma}) and (\ref{formula 20}), we have
\begin{align}\label{shuffle-sigma-iso} 
\widehat{\mathcal{SH}}_{k,n-k}(1)=\sum_{0\le m_{1}\le m_{2}\le \ldots \le m_{k}\le  n-k }^{}\left(\sigma^{m_{1}}D_{0}\sigma^{m_{2}-m_{1}}D_{0}\sigma^{m_{3}-m_{2}}\ldots \sigma^{m_{k}-m_{k-1}}D_{0}\right)(1).
\end{align}
\quad\\ 
Now, assume $\sigma=\id$. Note that $\widehat{\mathcal{SH}}_{0,n}(1)=1$ and
\begin{align*}
\sum_{0\le m_{1}\le m_{2}\le \ldots \le m_{k}\le  n-k }^{}1=\binom{n}{k}.
\end{align*}
Then the formula (\ref{shuffle-sigma-iso}) yields:
\begin{align*}
\widehat{\mathcal{SH}}_{k,n-k}(1)=\binom{n}{k}D_{0}^{k}(1)=\binom{n}{k}(\operatorname{ad}_{\id}x+y)^{k}(1).
\end{align*}
$(\operatorname{ad}_{\id}x+y)^{k}(1)$ is the Bell differential polynomial (see Formula (\ref{formula 32})). 
\end{remark}

\subsection{Binomial theorem given by $q$-Bell differential polynomials}\label{Section 3.2} \qquad

In Theorem \ref{thm:binom-formula-shuffle-poly-sigma-derivation}, if $A=\K\langle x,y\rangle$ and $\operatorname{ad}_{\sigma} x$ is a $q$-derivation, then Formula (\ref{shuffle-sigma-iso}) yields a $q$-version of the Bell differential polynomials, see Section \ref{Section 4.1}. First we recall the $q$-binomial coefficients, or Gaussian binomial coefficients. The $q$-analogue of integer $(n)_{q}$ for $n\in \mathbb{N}$ is defined by
$$
(n)_{q}=1+q+q^{2}+\cdots+q^{n-1}=
\left\{\begin{array}{ll}
\frac{1-q^{n}}{1-q} ,& q \neq 1, \\
n ,&  q=1,
\end{array}\right.{}
$$ 
and set $(0)_{q}=0$. The $q$-analogue of factorials $(n)_{q}!$ are defined as 
$$
(n)_{q}!=(1)_{q}(2)_{q}\ldots (n)_{q},
$$
and the $q$-binomial coefficients are defined as 
$$
\binom{n}{k}_{q}=\frac{(n)_{q}!}{(k)_{q}!(n-k)_{q}!}.
$$

We define a map $\sigma:\K\langle x,y\rangle\rightarrow \K \langle x,y\rangle$ by $\sigma(x)=qx$ and $\sigma(y)=qy$, which gives rise to an algebra isomorphism $\sigma:\K\langle x,y\rangle\rightarrow \K \langle x,y\rangle$. Then, for any homogeneous element $w\in \K\langle x,y\rangle$,
\begin{align*}
\sigma(w)=q^{|w|}w.
\end{align*}
In this case,  we  write $\operatorname{ad}_{q}x$ for  $\operatorname{ad}_{\sigma}x$.  Then
\begin{align*}
\operatorname{ad}_{q}x(w) &=[x,w]_{q}=xw-q^{|w|}wx,\\
D_{0} &=\operatorname{ad}_{q}x+y.
\end{align*}
It follows from Formula (\ref{shuffle-sigma-iso}) that 
\begin{align}\label{formula:SH-to-q-Bell-diff-poly}
\widehat{\mathcal{SH}}_{k,n-k}(1)=\sum_{0\le m_{1}\le m_{2}\le \ldots \le m_{k}\le  n-k }^{}q^{\sum_{i=1}^{k}m_{i}}(\operatorname{ad}_{q}x+y)^{k}(1).
\end{align}
From \cite[Theorem 3.1]{A1998} or \cite[Lemma (1)]{J1996}, we know 
\begin{align*}
\sum_{0\le m_{1}\le m_{2}\le \ldots \le m_{k}\le  n-k }^{}q^{\sum_{i=1}^{k}m_{i}}=\binom{n}{k}_{q}.
\end{align*}
It follows that Formula (\ref{formula:SH-to-q-Bell-diff-poly}) yields:
\begin{align}\label{q-Bell differential polynomials}
\widehat{\mathcal{SH}}_{k,n-k}(1)=\binom{n}{k}_{q}(\operatorname{ad}_{q}x+y)^{k}(1).
\end{align}
Formula (\ref{q-Bell differential polynomials}) can be used to define a $q$-analogue of the  Bell differential polynomials.

\begin{definition}
The \textit{q-Bell differential polynomials} $\widehat{B}_{n,q}:=\widehat{B}_{n,q}(x,y)$ in $\K\langle x,y\rangle$ are defined recursively by $\widehat{B}_{0,q}=1$ and 
\begin{align*}
\widehat{B}_{n,q}=(\operatorname{ad}_{q}x+y)^{n}(1), \quad n\ge 1.
\end{align*}
\end{definition}

\begin{corollary}\label{cor:binom-formula-by-q-Bell-diff-poly}
Assume $\sigma(w)=q^{|w|}w$ for any homogeneous element $w\in\K\langle x,y\rangle$. Then we have
\begin{align}
(x+y)^{n}=\sum_{k=0}^{n}\binom{n}{k}_{q}\widehat{B}_{k,q}(x,y)x^{n-k},
\end{align}
\begin{proof}
It follows from Theorem \ref{thm:binom-formula-shuffle-poly-sigma-derivation} and Formula (\ref{q-Bell differential polynomials}).
\end{proof}
\end{corollary}

\begin{example}
Let $y^{(0)}:=y$ and $y^{(n)}:=\operatorname{ad}_{q}x(y^{(n-1)})=xy^{(n-1)}-q^{n}y^{(n-1)}x$, $n\ge 1$.
Then
\begin{align*}
\widehat{B}_{0,q} &=1,\\
\widehat{B}_{1,q} &=y,\\
\widehat{B}_{2,q} &=y^{(1)}+y^{2},\\
\widehat{B}_{3,q} &=y^{(2)}+(2)_{q}yy^{(1)}+y^{(1)}y+y^{3},\\
\widehat{B}_{4,q} &=y^{(3)}+(3)_{q}yy^{(2)}+(3)_{q}(y^{(1)})^{2}+y^{(2)}y+(3)_{q}yy^{(1)}+(2)_{q}yy^{(1)}y+y^{(1)}y^{2}+y^{4},
\end{align*}
and
\begin{align*}
(x+y)^{2}&=x^{2}+(2)_{q}yx+(y^{(1)}+y^{2})=\widehat{B}_{0,q}x^{2}+(2)_{q}\widehat{B}_{1,q}x+\widehat{B}_{2,q},\\
(x+y)^{3}&=x^{3}+(3)_{q}yx^{2}+(3)_{q}y^{(1)}x+(3)_{q}y^{2}x+y^{(2)}+(2)_{q}yy^{(1)}+y^{(1)}y+y^{3}\\
& =\widehat{B}_{0,q}x^{3}+(3)_{q}\widehat{B}_{1,q}x^{2}+(3)_{q}\widehat{B}_{2,q}x+\widehat{B}_{3,q}.
\end{align*}
\end{example}

\begin{remark}\label{rmk:noncom-version-of-q-Bell-poly}
The \textit{partial q-Bell differential polynomials} $\widehat{B}_{n,k,q}$ in $\K\langle x,y\rangle$ can be recursively defined by $\widehat{B}_{n,k,q}:=0$ for $k=0$ or $k>n$, $\widehat{B}_{1,1,q}=\widehat{B}_{1,q}=y$ and 
\begin{align}\label{formula 25}
\widehat{B}_{n,k,q}=y\widehat{B}_{n-1,k-1,q}+\operatorname{ad}_{q}x(\widehat{B}_{n-1,k,q}).
\end{align}
Using Formula (\ref{formula 25}), we can prove that $\widehat{B}_{n,k,q}$ has the recursion: 
\begin{align}\label{q-Bell differential recursion}
\widehat{B}_{n,k,q}=\sum_{\ell=k-1}^{n-1}\binom{n-1}{\ell}_{q} \widehat{B}_{\ell,k-1,q}y^{(n-\ell)}.
\end{align}
We mention that the recursion (\ref{q-Bell differential recursion}) is equivalent to the recursion of the commutative partial $q$-Bell polynomials ${\bf B}_{n,k,q}$ introduced by Johnson \cite{J1996} in the case where $y^{(v)}y^{(u)}=y^{(u)}y^{(v)}$ for $0\le u<v$. Thus $\widehat{B}_{n,k,q}$ can be considered as the noncommutative version of ${\bf B}_{n,k,q}$.
\end{remark}

From the $q$-Bell differential polynomial $\widehat{B}_{n,k,q}$, we deduce the $q$-commutative Bell differential polynomials and the corresponding binomial theorem, see Subsection \ref{Section 5.2}.

\section{Bell differential polynomials and shuffle type polynomials}\label{Section 4}

From Sections \ref{Section 3}, we see that the shuffle type polynomials are related to the Bell differential polynomials and their $q$-analogues. Now we study more their relationship using the Lyndon-Shirshov basis.

\subsection{\bf Binomial theorem given by Bell differential polynomials}\label{Section 4.1} 
\qquad

First we recall the Bell differential polynomials.

\begin{definition}
{\cite{SR1998}}
\label{def:Bell-diff-poly}
Let $A$ be an associative algebra and $x,y\in A$. The \textit{Bell differential polynomials} $\widehat{B}_{n}=\widehat{B}_{n}[x,y]$ in $y$ with respect to $x$ and their dual $\widehat{B}_{n}^{*}=B_{n}^{*}[x,y]$ are defined inductively by
\begin{align*}
& \widehat{B}_{n+1}:=(\operatorname{ad}_{L}x)(\widehat{B}_{n})+y\widehat{B}_{n},\qquad \widehat{B}_{0}:=1\\
& \widehat{B}_{n+1}^{*}:=(\operatorname{ad}_{R}x)(\widehat{B}_{n}^{*})+\widehat{B}_{n}^{*}y,\qquad \widehat{B}_{0}^{*}:=1,
\end{align*}
where $\operatorname{ad}_{L}x,\operatorname{ad}_{R}x: A\rightarrow A$, $\operatorname{ad}_{L}x(w):=[x,w]=xw-wx$ and $\operatorname{ad}_{R}x(w):=[w,x]=wx-xw$.   The \textit{partial Bell differential polynomials} $\widehat{B}_{n,k}:=\widehat{B}_{n,k}(x,y)$ are the homogeneous parts of $\widehat{B}_{n}$ such that for any constant $\lambda$
\begin{align*}
\widehat{B}_{n}=\sum_{k=1}^{n}\widehat{B}_{n,k}, \qquad \widehat{B}_{n,k}(x,\lambda y)=\lambda^{k}\widehat{B}_{n,k}(x,y). 
\end{align*}
\end{definition}
\noindent Set $\widehat{B}_{n,k}:=0$ for $k=0$ or $k>n$. $\widehat{B}_{n,k}(E_{1},E_{2})$ has the recursion as follows (see the proof of \cite[Theorem 10]{SR1998}):
\begin{align}\label{formula 29}
\widehat{B}_{n,k}(E_{1},E_{2})=E_{2}\widehat{B}_{n-1,k-1}(E_{1},E_{2})+ \operatorname{ad}_{L} E_{1}(\widehat{B}_{n-1,k}(E_{1},E_{2})).
\end{align}

\begin{lemma}{\cite[Theorem 1]{SR1998}}\label{Lemma 4.2}
For every $x,y\in A$ and $n\in \mathbb{N}$, the following holds
\begin{align}
\label{binomial formula of Schimming-Rida} (x+y)^{n}&=\sum_{k=0}^{n}\binom{n}{k}\widehat{B}_{k}(x,y)x^{n-k}\\
&=\sum_{k=0}^{n}\binom{n}{k}x^{n-k}\widehat{B}_{k}^{*}(x,y).
\end{align}
\end{lemma}

\begin{remark}\label{Remark 4.3}
(a) By the definition and Formula (\ref{binomial formula of Schimming-Rida}), it is evident that $\widehat{B}_{n}=(\operatorname{ad}_{L}x+y)^{n}(1)$ and 
\begin{align}\label{formula 32}
(x+y)^{n}=\sum_{k=0}^{n}\binom{n}{k}(\operatorname{ad}_{L}x+y)^{k}(1)x^{n-k},\quad n\ge 0.
\end{align}
Let $y^{(0)}:=y$ and $y^{(k)}:=\operatorname{ad}_{L}x(y^{(k-1)})$ for $k\ge 1$. If $y^{(i)}y^{(j)}=y^{(j)}y^{(i)}$ for all $0\le i<j$, then $\widehat{B}_{n}$ yields  the classical Bell polynomials:
\begin{align}\label{classical Bell polynomials}
\widehat{B}_{n}(x,y)=\sum_{k_{1}+ \ldots +nk_{n}=n} \frac{n !}{k_{1}!\ldots k_{n}!}\left(\frac{y}{1 !}\right)^{k_{1}} \ldots\left(\frac{y^{(n-1)}}{n !}\right)^{k_{n}}.
\end{align}
Let $x^{(k)}:=(\operatorname{ad}_{R}y)^{k}(x)$ for $k\ge 0$. If $x^{(i)}x^{(j)}=x^{(j)}x^{(i)}$ for all $0\le i<j$, then $\widehat{B}_{n}^{*}$ has the same form:
\begin{align}\label{dual classical Bell polynomials}
\widehat{B}_{n}^{*}(y,x)=\sum_{k_{1}+\ldots +nk_{n}=n} \frac{n !}{k_{1}!\ldots k_{n}!}\left(\frac{x}{1 !}\right)^{k_{1}} \ldots\left(\frac{x^{(n-1)}}{n !}\right)^{k_{n}}.
\end{align}

\noindent It is worth noting that Formula (\ref{binomial formula of Schimming-Rida}) gives an answer to the question in \cite[Research problem 8.1]{MS2016}, which requires a binomial formula for the Ore extension $\mathbb{C}[y][x;\sigma=id,\delta]$. Note that $xy=yx+\delta(y)$ and $[\delta^{i}(y),\delta^{j}(y)]=0$ for $0\le i\le j$. It follows from Formulas (\ref{binomial formula of Schimming-Rida}) and (\ref{classical Bell polynomials}) that
\begin{align}\label{answer-for-case-sigma=id}
(x+y)^{n}=\sum_{k=0}^{n}\binom{n}{k}\left(\sum_{r_{1}+\ldots+kr_{k}=k}^{}\frac{k!}{r_{1}!\ldots r_{k}!}\left(\frac{y}{1!}\right)^{r_{1}} \ldots \left(\frac{\delta^{k-1}(y)}{k!}\right)^{r_{k}}\right)x^{n-k}.
\end{align}

\noindent
(b) Lemma \ref{Lemma 4.2} can be extended to the multinomial expansion in a natural way (see \cite[Proposition 1]{SR1998}):
\begin{align*}
(x+y_{2}+\ldots+y_{m})^{n} &=\sum_{k=0}^{n}\binom{n}{k}\widehat{B}_{k}(x,y_{2}+\ldots+y_{m})x^{n-k},
\end{align*}
where $\widehat{B}_{k}(x,y_{2}+\ldots+y_{m})=(\operatorname{ad}_{L}x+y_{2}+\ldots+y_{m})^{k}(1)$.
\end{remark}

Munthe-Kaas \cite{MK1995} introduced the equivalent form of the noncommutative Bell polynomials.  We follow \cite{LM2011,ELM2014} to recall them. Let $I=\{d_{j}\}_{j=1}^{\infty}$ be an infinite set and $\K\langle I \rangle$ the free algebra with the degree of $d_{j}$ being $j$. Suppose that $\partial: \K\langle I \rangle \rightarrow \K\langle I \rangle$ is a linear derivation defined by $\partial(d_{i})=d_{i+1}$ for $i\ge 1$ such that $\partial(w_{1}w_{2})=\partial(w_{1})w_2+w_{1}\partial(w_{2})$ for $w_{1},w_{2} \in \langle I\rangle$.
\begin{definition}
\label{def:Bell-poly-MK}
The \textit{noncommutative Bell polynomials} $\widetilde{B}_{n}$ are defined inductively by
\begin{align*}
& \widetilde{B}_{0}:=1,\\
& \widetilde{B}_{n}:=(d_{1}+\partial)(B_{n-1}),\quad n\ge 1,
\end{align*}
that is, $\widetilde{B}_{n}=(d_{1}+\partial)^{n}(1),n\ge 0$.
The \textit{partial Bell polynomial} $\widetilde{B}_{n,k}$ is the part of $\widetilde{B}_{n}$ consisting of terms of length $k$. 
\end{definition} 

It is clear that Definition \ref{def:Bell-diff-poly} and Definition \ref{def:Bell-poly-MK} are equivalent. Further, Lundervold--Munthe-Kaas considered the case: $d_{i}:=[\partial,d_{i-1}]=\partial d_{i-1}-d_{i-1}\partial$ for $i>1$, and gave the following binomial formula which is equivalent to Formula (\ref{binomial formula of Schimming-Rida}):
\begin{align}\label{formula:noncomm-binom-LMK}
\left(d_{1}+\partial\right)^{n}=\sum_{k=0}^{n}\binom{n}{k}\widetilde{B}_{k}\left(d_{1}, \ldots, d_{k}\right) \partial^{n-k}.
\end{align}
Moreover, an explicit expression of $\widetilde{B}_{n,k}$ can be found in \cite{LM2011,ELM2014}.

\subsection{Bell differential polynomials and shuffle type polynomials}\label{Section 4.2}
\qquad 

Recall the Bell differential polynomials $\widehat{B}_{n,k}(x,y)$ in Definition \ref{def:Bell-diff-poly}. Let $(X,\preceq)=(\{1,2\},1\prec 2)$, $E_{1}:=x$ and $E_{2}:=y$. Now we rewrite $\widehat{B}_{n,k}(E_{1},E_{2})$ as the linear combinations of products of the Lyndon-Shirshov basis elements. For example,
\begin{align*}
\widehat{B}_{1,1}(E_{1},E_{2})&=E_{2},\\
\widehat{B}_{2,1}(E_{1},E_{2})&=E_{12},\quad\quad\quad \widehat{B}_{2,2}(E_{1},E_{2})=E_{2}^{2},\\
\widehat{B}_{3,1}(E_{1},E_{2})&=E_{112}, \quad\quad\quad \widehat{B}_{3,2}(E_{1},E_{2})=2E_{2}E_{12}+E_{12}E_{2}=3E_{2}E_{12}+E_{122},\\
\widehat{B}_{3,3}(E_{1},E_{2})&=E_{2}^{3},\\
\widehat{B}_{4,1}(E_{1},E_{2})&=E_{1112}, \quad\quad\quad \widehat{B}_{4,2}(E_{1},E_{2})=3E_{2}E_{112}+3E_{12}^{2}+E_{112}E_{2}=4E_{2}E_{112}+3E_{12}^{2}+E_{1122},\\
\widehat{B}_{4,3}(E_{1},E_{2}) &=3E_{2}^{2}E_{12}+2E_{2}E_{12}E_{2}+E_{12}E_{2}^{2}=6E_{2}^{2}E_{12}+4E_{2}E_{122}+E_{1222},\\
\widehat{B}_{4,4}(E_{1},E_{2}) &=E_{2}^{4}.
\end{align*}
From Examples \ref{eg.:SH3} and \ref{eg.:SH4} , it is not difficult to find that the relation between the Bell differential polynomials and the shuffle type polynomials (in terms of the Lyndon-Shirshov basis) is as follows. 
\begin{theorem}\label{bell-shuffle} 
Let $(X,\preceq)=(\{1,2\},1\prec 2)$ and $\mathcal{SH}_{i,j}:=\mathcal{SH}_{i,j}(E_{2},E_{1})$ for $i,j\ge 0$. \\
$\mathrm(a)$ Let $\mathcal{SH}_{i,j}^{[2]}$ be the sum of all the terms in $\mathcal{SH}_{i,j}$ of which the rightmost element is not $E_{1}$. Then $\widehat{B}_{n,k}(E_{1},E_{2})=\mathcal{SH}_{k,n-k}^{[2]}$.\\
$\mathrm(b)$ Dually, denote by $\mathcal{SH}_{i,j}'$ the sum of all the terms in $\mathcal{SH}_{i,j}$ of which the leftmost element is not $E_{2}$. Then $\widehat{B}_{n,k}^{*}(E_{2},E_{1})=\mathcal{SH}_{n-k,k}'$.
\end{theorem} 

\begin{proof}
Let $\mathcal{SH}_{i,j}^{[1]}:=\mathcal{SH}_{i,j}-\mathcal{SH}_{i,j}^{[2]}$. 
Let $f_{1}$ be a term of $\mathcal{SH}_{i,j}^{[2]}$. By Formula (\ref{formula:commutators-x-LSbasis}), $f_{1}$ must not be a term of $E_{2}\mathcal{SH}_{i-1,j}^{[1]}+E_{1}\mathcal{SH}_{i-1,j}^{[1]}$. It follows from Formulas (\ref{formula:recursion1-shuffle-poly}) that $f_{1}$ only appears in $E_{2}\mathcal{SH}_{i-1,j}^{[2]}+E_{1}\mathcal{SH}_{i,j-1}^{[2]}$, or equivalently, in $E_{2}\mathcal{SH}_{i-1,j}^{[2]}+\left[E_{1},\mathcal{SH}_{i,j-1}^{[2]}\right]$ by Formula (\ref{formula:commutators-x-LSbasis}). Conversely, using Formula (\ref{formula:commutators-x-LSbasis}) we see that the leftmost element of a term $f_{2}$ in $E_{2}\mathcal{SH}_{i-1,j}^{[2]}+\left[E_{1},\mathcal{SH}_{i,j-1}^{[2]}\right]$ is not $E_{1}$. Thus, $f_{2}$ must appear in $\mathcal{SH}_{i,j}^{[2]}$ from Formula (\ref{formula:recursion1-shuffle-poly}). So we obtain the following recursion 
\begin{align*}
\mathcal{SH}_{i,j}^{[2]}=E_{2}\mathcal{SH}_{i-1,j}^{[2]}+\left[E_{1},\mathcal{SH}_{i,j-1}^{[2]}\right], \qquad i,j\ge 0.
\end{align*}
which is equivalent to the recursion of $\widehat{B}_{i+j,i}(E_{1},E_{2})$ (see Formula (\ref{formula 29})). Therefore, $\widehat{B}_{n,k}(E_{1},E_{2})=\mathcal{SH}_{k,n-k}^{[2]}$ from the initial conditions that $\widehat{B}_{0,0}(E_{1},E_{2})=1=\mathcal{SH}_{0,0}^{[2]}$ and $\widehat{B}_{1,1}(E_{1},E_{2})=E_{2}=\mathcal{SH}_{1,0}^{[2]}$.

Similar to the proof of Part (a), using Formula  (\ref{formula:recursion2-shuffle-poly}) and Corollary \ref{cor:commutator-LS-basis} (d), we obtain that $\mathcal{SH}_{i,j}'$ has the following recursion:
\begin{align}
\mathcal{SH}_{i,j}'=\left[\mathcal{SH}_{i-1,j}',E_{2}\right]+\mathcal{SH}_{i,j-1}'E_{1}.
\end{align}
Therefore, $\widehat{B}_{n.k}^{*}(E_{2},E_{1})=\mathcal{SH}_{n-k,k}'$.
\end{proof}

Now we use Theorem \ref{thm:noncom-binom-coeff-LSbasis} to describe $\widehat{B}_{n,k}(E_{1},E_{2})$ and $\widehat{B}_{n}(E_{1},E_{2})$ in terms of the Lyndon-Shirshov basis, in a generalized form of the   Bell polynomials (\ref{classical Bell polynomials}).

\begin{corollary}\label{cor:Bell-diff-poly-by-LS-basis}
The following hold in $\K\langle E_{1},E_{2}\rangle$:
\begin{align}\label{formula 37}
&\widehat{B}_{n,k}(E_{1},E_{2}) 
=\sum_{t_{2}+\sum_{i=1}^{r}t_{\alpha_{i}}\alpha_{i}(2)=k \atop \sum_{i=1}^{r}t_{\alpha_{i}}\alpha_{i}(1)=n-k}^{}\frac{k!}{t_{2}!t_{\alpha_{1}}!\cdots t_{\alpha_{r}}!}E_{2}^{t_{2}}\left(\frac{C_{E_{\alpha_{1}}}E_{\alpha_{1}}}{|\alpha_{1}|!}\right)^{t_{\alpha_{1}}} \ldots \left(\frac{C_{E_{\alpha_{r}}}E_{\alpha_{r}}}{|\alpha_{r}|!}\right)^{t_{\alpha_{r}}},\\
\noindent \label{formula 38}
&\widehat{B}_{n}(E_{1},E_{2}) 
=\sum_{t_{2}+\sum_{i=1}^{r}t_{\alpha_{i}}|\alpha_{i}|=n}^{}\frac{k!}{t_{2}!t_{\alpha_{1}}!\cdots t_{\alpha_{r}}!}E_{2}^{t_{2}}\left(\frac{C_{E_{\alpha_{1}}}E_{\alpha_{1}}}{|\alpha_{1}|!}\right)^{t_{\alpha_{1}}} \ldots \left(\frac{C_{E_{\alpha_{r}}}E_{\alpha_{r}}}{|\alpha_{r}|!}\right)^{t_{\alpha_{r}}},
\end{align}
where $2\succ \alpha_{1}\succ \alpha_{2}\succ \ldots \succ \alpha_{r} \succ1$.
\end{corollary}

\begin{proof}
Note that $C_{E_{1}}=C_{E_{2}}=1$. The claim follows from Theorem \ref{bell-shuffle} (a) and Formula (\ref{j-z formula 2}).
\end{proof}

\begin{remark}\label{Remark 4.8}
Write 
$$\omega_{n}:=\underbrace{1\cdots1}_{n}2,\qquad n\ge 0.$$ Let
$\mathcal{I}$ be the ideal of $\mathcal{T}=\K\langle E_{1},E_{2}\rangle$ generated by $\{E_{\alpha}|~\alpha(2)\ge 2\}$. 
Note that $C_{E_{\omega_{n}}}=1$ for $n\ge 1$ by Corollary \ref{cor:particular-case-noncom-binom} (a). Thus  Corollary \ref{cor:Bell-diff-poly-by-LS-basis} derives the Bell polynomials (\ref{classical Bell polynomials}) in the quotient $\mathcal{T}/\mathcal{I}$ of $\mathcal{T}$ . 
\end{remark} 

In the multinomial case, Theorem \ref{bell-shuffle} can be generalized as follows.
\begin{corollary}
Let $(X,\preceq)=(\{1,\ldots,m\},1\prec \ldots \prec m)$. In  $\K\langle E_{1},\ldots, E_{m}\rangle$, $\widehat{B}_{n,k}(E_{1},E_{2}+\ldots+E_{m})$ is equal to the sum of the terms in
$\sum_{i_{m}+\ldots +i_{2}=k}^{} \mathcal{SH}_{i_{m}, \ldots, i_{2},n-k}(E_{m},\ldots,E_{2},E_{1})$ of which the rightmost element is not $E_{1}$.
\end{corollary}

Formulas (\ref{formula 37}) and (\ref{formula 38}) can be generalized to $\widehat{B}_{n,k}(E_{1},E_{2}+\ldots+E_{m})$ and $\widehat{B}_{n}(E_{1},E_{2}+\ldots+E_{m})$. We use Corollary \ref{cor:noncom-multinomial-thm} to generalize Corollary \ref{cor:Bell-diff-poly-by-LS-basis} to the following.
\begin{corollary}\label{Corollary 4.10}
Let $(X,\preceq)=(\{1,\ldots,m\},1\prec \ldots \prec m)$. The following hold in $\K\langle E_{1},\ldots,E_{m}\rangle$:
\begin{align*}
\widehat{B}_{n,k}(E_{1},E_{2}+\ldots+E_{m})
=& \sum_{\sum_{i=1}^{r}\sum_{j=2}^{m}t_{\alpha_{i}}\alpha_{i}(j)=k \atop \sum_{i=1}^{r}t_{\alpha_{i}}\alpha_{i}(1)=n-k}^{}\frac{k!}{t_{\alpha_{1}}!\ldots t_{\alpha_{r}}!}\left(\frac{C_{E_{\alpha_{1}}}E_{\alpha_{1}}}{|\alpha_{1}|!}\right)^{t_{\alpha_{1}}} \ldots \left(\frac{C_{E_{\alpha_{r}}}E_{\alpha_{r}}}{|\alpha_{r}|!}\right)^{t_{\alpha_{r}}},\\
\widehat{B}_{n}(E_{1},E_{2}+\ldots+E_{m})
=& \sum_{\sum_{i=1}^{r}t_{\alpha_{i}}|\alpha_{i}|=n}^{}\frac{k!}{t_{\alpha_{1}}!\ldots t_{\alpha_{r}}!}\left(\frac{C_{E_{\alpha_{1}}}E_{\alpha_{1}}}{|\alpha_{1}|!}\right)^{t_{\alpha_{1}}} \ldots \left(\frac{C_{E_{\alpha_{r}}}E_{\alpha_{r}}}{|\alpha_{r}|!}\right)^{t_{\alpha_{r}}},
\end{align*}
where $m\succeq \alpha_{1}\succ \ldots \succ \alpha_{r} \succ 1$.
\end{corollary}

\section{ Applications}

\subsection{Applications of Theorem \ref{thm:noncom-binom-coeff-LSbasis}}\label{Section 5.1}
\qquad 

In this subsection, we give applications of Theorem \ref{thm:noncom-binom-coeff-LSbasis}, which contain some classical binomial formulas. 
\begin{example}\label{Example 5.1}
(a) Let $(X,\prec)=(\{1,2\}, 1\prec 2)$. Assume $E_{12}=0$, that is, $E_{1}E_{2}=E_{2}E_{1}$.  The set of the generators of the Lyndon-Shirshov basis  in the quotient $\K\langle X\rangle/(E_{12})$ is $\{E_{1},E_{2}\}$. In Theorem \ref{thm:noncom-binom-coeff-LSbasis}, let $n=2,~\alpha_{1}=2,~\alpha_{2}=1,~t_{\alpha_{1}}=d-k$ and $t_{\alpha_{2}}=k$. Since $|\alpha_{1}|=|2|=1,~|\alpha_{2}|=|1|=1,~C_{E_{2}}=1$ and $C_{E_{1}}=1$, we have: 
\begin{align*}
C_{E_{2}^{d-k}E_{1}^{k}}=\frac{d!}{(d-k)!k!},
\end{align*} which is the binomial coefficient in the commutative case.\\
\qquad\\
(b) Let $(X,\prec)=(\{1,2,\ldots,m\},1\prec 2\prec \ldots \prec m)$. Assume that $E_{ij}=0$ for all $1\le i<j\le m$, that is, $E_{i}E_{j}=E_{j}E_{i}$. The set of the generators of the Lyndon-Shirshov basis is $\{E_{1},E_{2},\ldots,E_{m}\}$. Set $n=m$, $\alpha_{1}=m$, $\alpha_{2}=m-1$, \ldots, $\alpha_{m}=1$ in Corollary \ref{cor:noncom-multinomial-thm}. 
Note that $|\alpha_{\ell}|=1$ and $C_{E_{\alpha_{\ell}}}=1$ for $1\le \ell\le m$. Thus $$C_{E_{m}^{t_{m}}E_{m-1}^{t_{m-1}}\ldots E_{1}^{t_{1}}}=\frac{(t_{m}+t_{m-1}+\ldots +t_{1})!}{t_{m}!t_{m-1}!\ldots t_{1}!},$$ which is the well-known multinomial coefficient in the commutative case.
\end{example}

\begin{example}\label{eg.:quasi-commutative}
Let $(X,\prec)=(\{1,2\},1\prec 2)$. If $E_{112}=0$ and $E_{122}=0$, that is,
$$E_{1}E_{2}=E_{2}E_{1}+E_{12}, \qquad E_{1}E_{12}=E_{12}E_{1}, \qquad E_{12}E_{2}=E_{2}E_{12}.$$
In this case, $\{E_{2},E_{12},E_{1}\}$ generates the Lyndon-Shirshov basis. Let $n=3,~\alpha_{1}=2,~\alpha_{2}=12,~\alpha_{3}=1$ in Theorem \ref{thm:noncom-binom-coeff-LSbasis}. Observe that $|2|=1,~|12|=2,~|1|=1,~C_{E_{2}}=1,~C_{E_{12}}=1$ and $C_{E_{1}}=1$. Thus
\begin{align*}
C_{E_{2}^{t_{2}}E_{12}^{t_{12}}E_{1}^{t_{1}}}=\frac{(t_{2}+2t_{12}+t_{1})!}{t_{2}!(2!)^{t_{12}}t_{12}!t_{1}!}, 
\end{align*}
Therefore, for $d\ge 0$,
\begin{align}\label{formula 39}
(E_{1}+E_{2})^{d}=\sum_{t_{2}+2t_{12}+t_{1}=d} \frac{d!}{t_{2}!(2!)^{t_{12}}t_{12}!t_{1}!} E_{2}^{t_{2}}E_{12}^{t_{12}}E_{1}^{t_{1}}.
\end{align}
\end{example}

\begin{remark}\label{rmk:quasi-comm-case}
Replace $E_{1}$ by $x$, $E_{2}$ by $y$ and $E_{12}$ by $h$ in Formula (\ref{formula 39}). Then (\ref{formula 39}) has the equivalent form as the binomial formula (\ref{formula-1}). Note that the binomial formula (\ref{formula:Blumen}) is a $q$-analogue of Formula (\ref{formula 39}). We generalize (\ref{formula:Blumen}) in Section \ref{Section 5.2}, using the $q$-commutative Bell differential polynomials.
\end{remark}

Moreover, we can consider the cases that retain more Lyndon-Shirshov basis elements, not only the commutative case as Example \ref{Example 5.1} and the quasi-commutative case as Example \ref{eg.:quasi-commutative}.
\begin{example}
Let $(X,\prec)=(\{1,2\},1\prec 2)$. Let $E_{1112}=E_{1122}=E_{1222}=E_{11212}=E_{12122}=0$, that is,
\begin{align*}
E_{1}E_{2} &=E_{2}E_{1}+E_{12}, \quad \quad  E_{1}E_{12} =E_{12}E_{1}+E_{112}, \quad\quad E_{12}E_{2}=E_{2}E_{12}+E_{122},\\
E_{1}E_{112} &=E_{112}E_{1}, \quad \quad \quad \quad  E_{1}E_{122}=E_{122}E_{1}, \quad\quad\quad\quad E_{112}E_{2}=E_{2}E_{112},\\
E_{122}E_{2} &=E_{2}E_{122}, \quad \quad \quad \quad E_{112}E_{12} =E_{12}E_{112},  \quad \quad \quad E_{12}E_{122}=E_{122}E_{12},\\
 E_{112}E_{122} &=E_{122}E_{112}.
\end{align*}  
In this case, the generator set of Lyndon-Shirshov basis is $\{E_{2},E_{122},E_{12},E_{112},E_{1}\}$.
Let $n=5,~\alpha_{1}=2,~\alpha_{2}=122,~\alpha_{3}=12,~\alpha_{4}=112$ and $\alpha_{5}=1$ in Theorem \ref{thm:noncom-binom-coeff-LSbasis}. Then $C_{E_{2}}=1,~C_{E_{122}}=1,~C_{E_{12}}=1,~C_{E_{112}}=1$ and $C_{E_{1}}=1$. Thus we obtain
\begin{align*}
C_{E_{2}^{t_{2}}E_{122}^{t_{122}}E_{12}^{t_{12}}E_{112}^{t_{112}}E_{1}^{t_{1}}}=\frac{(t_{2}+3t_{122}+2t_{12}+3t_{112}+t_{1})!}{t_{2}!(3!)^{t_{122}}t_{122}!(2!)^{t_{12}}t_{12}!(3!)^{t_{112}}t_{112}!t_{1}!}.
\end{align*}
Therefore, for $d\ge 0$,
\begin{align*}
(E_{1}+E_{2})^{d}=\sum_{t_{2}+3t_{122}+2t_{12}+3t_{112}+t_{1}=d} \frac{d!}{t_{2}!(3!)^{t_{122}}t_{122}!(2!)^{t_{12}}t_{12}!(3!)^{t_{112}}t_{112}!t_{1}!}  E_{2}^{t_{2}}E_{122}^{t_{122}}E_{12}^{t_{12}}E_{112}^{t_{112}}E_{1}^{t_{1}}.\\
\end{align*}

\end{example}

\begin{example}
Let $(X,\prec)=(\{1,2,3\},1\prec 2\prec 3)$. Assume $E_{112}=E_{113}=E_{122}=E_{123}=E_{132}=E_{133}=E_{223}=E_{233}=0$. Then $[E_{12},E_{3}]=E_{123}+E_{132}=0$. Thus we have the relations:
\begin{align*}
E_{1}E_{2} &=E_{2}E_{1}+E_{12}, \quad \quad E_{1}E_{3}=E_{3}E_{1}+E_{13}, \qquad E_{2}E_{3}=E_{3}E_{2}+E_{23},\\
E_{1}E_{12} &=E_{12}E_{1}, \quad \quad \quad \quad E_{1}E_{13}=E_{13}E_{1}, \quad\quad\quad\quad E_{1}E_{23}=E_{23}E_{1},\\
E_{12}E_{2} &=E_{2}E_{12}, \quad\quad\quad\quad E_{12}E_{3}=E_{3}E_{12}, \quad\quad\quad\quad E_{13}E_{2}=E_{2}E_{13}\\
E_{13}E_{3} &=E_{3}E_{13}, \quad\quad\quad\quad E_{2}E_{23}=E_{23}E_{2}, \quad\quad\quad\quad E_{23}E_{3}=E_{3}E_{23},\\
E_{12}E_{13} &=E_{13}E_{12}, ~\quad\quad\quad E_{12}E_{23}=E_{23}E_{12}, \quad\quad\quad E_{13}E_{23}=E_{23}E_{13}.
\end{align*}
The generator set of Lyndon-Shirshov basis is $\{E_{3},E_{23},E_{2},E_{13},E_{12},E_{1}\}$.
Let $n=6,~\alpha_{1}=3,~\alpha_{2}=23,~\alpha_{3}=2,~\alpha_{4}=13,~\alpha_{5}=12,~\alpha_{6}=1$ in Corollary \ref{cor:noncom-multinomial-thm}. Then $C_{E_{3}}=C_{E_{23}}=C_{E_{2}}=C_{E_{13}}=C_{E_{12}}=C_{E_{1}}=1$. Thus we obtain 
\begin{align*}
C_{E_{3}^{t_{3}}E_{23}^{t_{23}}E_{2}^{t_{2}}E_{13}^{t_{13}}E_{12}^{t_{12}}E_{1}^{t_{1}}}=\frac{(t_{3}+2t_{23}+t_{2}+2t_{13}+2t_{12}+t_{1})!}{t_{3}!(2!)^{t_{23}}t_{23}!t_{2}!(2!)^{t_{13}}t_{13}!(2!)^{t_{12}}t_{12}!t_{1}!}.
\end{align*}
 Therefore, for $d\ge 0$,
\begin{align*}
(E_{1}+E_{2}+E_{3})^{d}=\sum_{t_{3}+2t_{23}+t_{2}+2t_{13}+2t_{12}+t_{1}=d} \frac{d!}{t_{3}!(2!)^{t_{23}}t_{23}!t_{2}!(2!)^{t_{13}}t_{13}!(2!)^{t_{12}}t_{12}!t_{1}!} E_{3}^{t_{3}}E_{23}^{t_{23}}E_{2}^{t_{2}}E_{13}^{t_{13}}E_{12}^{t_{12}}E_{1}^{t_{1}}.
\end{align*}
\end{example}

\subsection{Applications of $q$-Bell differential polynomials}\label{Section 5.2} 
\qquad

In this subsection, we deduce a $q$-commutative form from the $q$-Bell differential polynomial $\widehat{B}_{n,k,q}$. For example, if $y^{(1)}y=q^{2}yy^{(1)}$ and $y^{(2)}y=q^{3}yy^{(2)}$, then 
\begin{align*}
\widehat{B}_{3,2,q} &=(3)_{q}yy^{(1)},\\
\widehat{B}_{4,2,q} &=(4)_{q}yy^{(2)},\\
\widehat{B}_{4,3,q} &=\binom{4}{2}_{q}y^{2}y^{(1)}.
\end{align*}
In general, if we assume $y^{(v-1)}y^{(u-1)}=q^{v}y^{(u-1)}y^{(v-1)}$ for all $1\le u<v$, then we have the following formulas.

\begin{theorem}\label{thm:q-comm-Bell-diff-poly}
Let $d_{i}:=y^{(i-1)}$ for $i\ge 1$. Assume $d_{v}d_{u}=q^{v}d_{u}d_{v}$ for all $1\le u<v$. Then the following hold
\begin{align}
\label{formula 42}
\widehat{B}_{n,k,q}&=\sum_{t_{1}+t_{2}+\ldots+t_{n-k+1}=k \atop t_{1}+2t_{2}+\ldots+(n-k+1)t_{n-k+1}=n}^{}\frac{(n)_{q}!}{\prod_{i=1}^{n-k+1}((i)_{q}!)^{t_{i}}(t_{i})_{q^{i}}!}d_{1}^{t_{1}}d_{2}^{t_{2}}\cdots d_{n-k+1}^{t_{n-k+1}},\\
\label{formula 43}
\widehat{B}_{n,q}&=\sum_{t_{1}+2t_{2}+\ldots+nt_{n}=n}^{}\frac{(n)_{q}!}{\prod_{i=1}^{n}((i)_{q}!)^{t_{i}}(t_{i})_{q^{i}}!}d_{1}^{t_{1}}d_{2}^{t_{2}}\cdots d_{n}^{t_{n}},
\end{align}
\end{theorem}

\begin{proof}
Denote by $C_{d_{i_{1}}^{t_{i_{1}}}d_{i_{2}}^{t_{i_{2}}}\cdots d_{i_{m}}^{t_{i_{m}}}}$ the coefficient of the term $d_{i_{1}}^{t_{i_{1}}}d_{i_{2}}^{t_{i_{2}}}\cdots d_{i_{m}}^{t_{i_{m}}}$ in $\widehat{B}_{n+1,k+1,q}$, where $1\le i_{1}< i_{2}<\ldots<i_{m}\le n-k+1$ and $t_{i_{1}},\ldots,t_{i_{m}}\ge 1$. It follows from Formula (\ref{q-Bell differential recursion}) that the term $d_{i_{1}}^{t_{i_{1}}}d_{i_{2}}^{t_{i_{2}}}\cdots d_{i_{m}}^{t_{i_{m}}}$ occurs in the polynomials $(1\le j\le m)$:
\begin{align*}
\binom{n}{n-i_{j}+1}_{q}\left(C_{d_{i_{1}}^{t_{i_{1}}}d_{i_{2}}^{t_{i_{2}}}\cdots d_{i_{(j-1)}}^{t_{i_{(j-1)}}}d_{i_{j}}^{t_{i_{j}}-1}d_{i_{(j+1)}}^{t_{i_{(j+1)}}}\cdots d_{i_{m}}^{t_{i_{m}}}}
d_{i_{1}}^{t_{i_{1}}}d_{i_{2}}^{t_{i_{2}}}\cdots d_{i_{(j-1)}}^{t_{i_{(j-1)}}}d_{i_{j}}^{t_{i_{j}}-1}d_{i_{(j+1)}}^{t_{i_{(j+1)}}}\cdots d_{i_{m}}^{t_{i_{m}}}\right)d_{i_{j}},
\end{align*}
where $d_{i_{1}}^{t_{i_{1}}}d_{i_{2}}^{t_{i_{2}}}\cdots d_{i_{(j-1)}}^{t_{i_{(j-1)}}}d_{i_{j}}^{t_{i_{j}}-1}d_{i_{(j+1)}}^{t_{i_{(j+1)}}}\cdots d_{i_{m}}^{t_{i_{m}}}$ is a term of $\widehat{B}_{n-i_{j}+1,k,q}$. Note that $d_{v}d_{u}=q^{v}d_{u}d_{v}$, $(rs)_{q}=(r)_{q}(s)_{q^{r}}$ and $(r+s)_{q}=(r)_{q}+q^{r}(s)_{q}$. By induction on $n$, we have the following equations:
\begin{align*}
C_{d_{i_{1}}^{t_{i_{1}}}d_{i_{2}}^{t_{i_{2}}}\cdots d_{i_{m}}^{t_{i_{m}}}} &=\sum_{j=1}^{m}\binom{n}{n-i_{j}+1}_{q}\frac{(n-i_{j}+1)_{q}!(i_{j})_{q}!(t_{i_{j}})_{q^{i_{j}}}}{\prod_{r=1}^{m}((i_{r})_{q}!)^{t_{i_{r}}}(t_{i_{r}})_{q^{i_{r}}}!}q^{(i_{m}t_{i_{m}}+i_{m-1}t_{i_{m-1}}+\ldots+i_{(j+1)}t_{i_{(j+1)}})}\\
&=\frac{(n)_{q}!\sum_{j=1}^{m}(i_{j})_{q}(t_{i_{j}})_{q^{i_{j}}}q^{(i_{m}t_{i_{m}}+i_{m-1}t_{i_{m-1}}+\ldots+i_{(j+1)}t_{i_{(j+1)}})}}{\prod_{r=1}^{m}((i_{r})_{q}!)^{t_{i_{r}}}(t_{i_{r}})_{q^{i_{r}}}!}\\
&=\frac{(n)_{q}!\sum_{j=1}^{m}(i_{j}t_{i_{j}})_{q}q^{(i_{m}t_{i_{m}}+i_{m-1}t_{i_{m-1}}+\ldots+i_{(j+1)}t_{i_{(j+1)}})}}{\prod_{r=1}^{m}((i_{r})_{q}!)^{t_{i_{r}}}(t_{i_{r}})_{q^{i_{r}}}!}\\
&=\frac{(n)_{q}!\left((i_{m}t_{i_{m}})_{q}+q^{i_{m}t_{i_{m}}}(i_{m-1}t_{i_{m-1}})_{q}+\ldots+q^{(i_{m}t_{i_{m}}+i_{m-1}t_{i_{m-1}}+\ldots+i_{1}t_{i_{1}})}(i_{1}t_{i_{1}})_{q}\right)}{\prod_{r=1}^{m}((i_{r})_{q}!)^{t_{i_{r}}}(t_{i_{r}})_{q^{i_{r}}}!}\\
&=\frac{(n)_{q}!(\sum_{j=1}^{m}i_{j}t_{i_{j}})_{q}}{\prod_{r=1}^{m}((i_{r})_{q}!)^{t_{i_{r}}}(t_{i_{r}})_{q^{i_{r}}}!}\\
&=\frac{(n+1)_{q}!}{\prod_{r=1}^{m}((i_{r})_{q}!)^{t_{i_{r}}}(t_{i_{r}})_{q^{i_{r}}}!}.
\end{align*}
\end{proof}

Note that if $q=1$, then Formulas (\ref{formula 42}), (\ref{formula 43}) derive the classical Bell polynomials. For this reason, We call $\widehat{B}_{n,q}$ the $q$-\textit{commutative Bell (differential) polynomials}.

From Corollary \ref{cor:binom-formula-by-q-Bell-diff-poly} and Theorem \ref{thm:q-comm-Bell-diff-poly}, we have the following conclusion. 
\begin{corollary}\label{cor:binom-q-comm-Bell-diff-poly}
Assume $y^{(v-1)}y^{(u-1)}=q^{v}y^{(u-1)}y^{(v-1)}$ for $v\ge u\ge 1$. Then 
\begin{align}
(x+y)^{n} &=\sum_{k=0}^{n}\binom{n}{k}_{q}\left(\sum_{t_{1}+2t_{2}+\ldots+kt_{k}=k}^{}\frac{(k)_{q}!}{\prod_{i=1}^{k}((i)_{q}!)^{t_{i}}(t_{i})_{q^{i}}!}y^{t_{1}}(y^{(1)})^{t_{2}}\cdots (y^{(k-1)})^{t_{k}}\right)x^{n-k},
\end{align}
or equivalently,

\begin{align}\label{formula 45}
(x+y)^{n} =\sum_{t_{1}+2t_{2}+\ldots+nt_{n}+t=n}^{}\frac{(n)_{q}!}{\prod_{i=1}^{n}((i)_{q}!)^{t_{i}}(t_{i})_{q^{i}}!(t)_{q}!}y^{t_{1}}(y^{(1)})^{t_{2}}\cdots (y^{(n-1)})^{t_{n}}x^{t}.
\end{align}
\end{corollary}

\begin{remark}\label{rmk:q-Bell-diff-poly-to-Blumen-formula}
Let $q\neq 0$. Assume $xy=qyx+y^{(1)},xy^{(1)}=q^{2}y^{(1)}x$ and $y^{(1)}y=q^{2}yy^{(1)}$. Then $y^{(n)}=0$ for $n\ge 2$. The formula (\ref{formula 45}) yields 
\begin{align*}
(x+y)^{n} 
&=\sum_{r+2s+t=n}^{}\frac{(n)_{q}!}{(r)_{q}!(2)_{q}^{s}(s)_{q^{2}}(t)_{q}!}y^{r}(y^{(1)})^{s}x^{t},
\end{align*}
which is equivalent to Formula (\ref{formula:Blumen}).
\end{remark}

\subsection{Bialgebras or Hopf algebras related to the shuffle type polynomials}\label{Section 5.3}
\qquad

We apply the shuffle type polynomials to the construction of bialgebras or Hopf algebras.

\subsubsection{The noncommutative F\`aa di Bruno bialgebra}
\qquad

From the noncommutative formal diffeomorphisms, Brouder--Frabetti--Krattenthaler \cite{BFK2006} introduced the free noncommutative Fa\`a di Bruno bialgebra of which the coproduct are determined by the following polynomials in the free algebra $\K\langle \{a_{k}\}_{k=0}^{\infty} \rangle$:
\begin{align*}
Q_{m}^{(n)}(a)=\sum_{i_{0}+i_{1}+\ldots+i_{n}=m \atop
i_{0},i_{1},\ldots,i_{n}\ge 0}  a_{i_{0}}a_{i_{1}}\cdots a_{i_{n}}, \qquad \forall~m,n\ge 0,
\end{align*}
see \cite[Section 2.2]{BFK2006}.

We show that the shuffle type polynomials are strongly related to $Q_{m}^{(n)}(a)$ as follows.
\begin{theorem}\label{thm:shuffle-poly-BFdBnc-coproduct}
Let $(X,\prec)=(\{g,h\}, g\prec h)$ and $a_{k}:=g\underbrace{h\cdots h}_{k}\in \X$ for $k\ge 0$. Then we have:
\begin{align}
g\mathcal{SH}_{m,n}(h,g)=\sum_{i_{0}+i_{1}+\ldots+i_{n}=m,\atop i_{0},i_{1},\ldots,i_{n}\ge 0}^{}a_{i_{0}}a_{i_{1}}\cdots a_{i_{n}}, \qquad \forall~ m,n\ge 0.
\end{align}
\end{theorem}

\begin{proof}
By Formula (\ref{formula:recursion2-shuffle-poly}) and induction on $m+n$, we have
\begin{align*}
g\mathcal{SH}_{m,n}(h,g)&=g\mathcal{SH}_{m,n-1}(h,g)g+g\mathcal{SH}_{m-1,n}(h,g)h\\
=&\sum_{i_{0}+i_{1}+\ldots+i_{n-1}=m \atop i_{0},i_{1},\ldots,i_{n-1}\ge 0}^{}a_{i_{0}}a_{i_{1}}\cdots a_{i_{n-1}}g+\sum_{j_{0}+j_{1}+\ldots+j_{n}=m-1,\atop j_{0},j_{1},\ldots,j_{n}\ge 0}^{}a_{j_{0}}a_{j_{1}}\cdots a_{j_{n-1}}a_{j_{n}}h\\
=&\sum_{i_{0}+i_{1}+\ldots+i_{n-1}=m \atop i_{0},i_{1},\ldots,i_{n-1}\ge 0}^{}a_{i_{0}}a_{i_{1}}\cdots a_{i_{n-1}}a_{0}+\sum_{j_{0}+j_{1}+\ldots+j_{n}=m-1,\atop j_{0},j_{1},\ldots,j_{n}\ge 0}^{}a_{j_{0}}a_{j_{1}}\cdots a_{j_{n-1}}a_{(j_{n}+1)}\\
=&\sum_{i_{0}+i_{1}+\ldots+i_{n}=m \atop i_{0},i_{1},\ldots,i_{n-1}\ge 0 ,~i_{n}=0}a_{i_{0}}a_{i_{1}}\cdots a_{i_{n}}+\sum_{i_{0}+i_{1}+\ldots+i_{n}=m \atop i_{0},i_{1},\ldots,i_{n-1}\ge 0,~i_{n}\ge 1}^{}a_{i_{0}}a_{i_{1}}\cdots a_{i_{n}}\\
=& \sum_{i_{0}+i_{1}+\ldots+i_{n}=m \atop i_{0},i_{1},\ldots,i_{n}\ge 0}^{}a_{i_{0}}a_{i_{1}}\cdots a_{i_{n}}.
\end{align*}
\end{proof}

\subsubsection{Pointed bialgebras or Hopf algebras}
\qquad

The following proposition concerning the bialgebras or Hopf algebras obtained from the shuffle type polynomials is due to Radford \cite{R1977}.
\begin{proposition}\label{pro:bialgebra-shuffle-relations}
Let $\mathcal{SH}_{n}$ $(n\ge 2)$ be  the algebra over $\K$  generated by $g, h$  satisfying the shuffle type  relations:
\begin{align*}
& \mathcal{SH}_{n-i,i}(h,g)=0,\quad 1\le i \le n-1.
\end{align*}
Then $\mathcal{SH}_{n}$ is an $\mathbb{N}$-graded pointed bialgebra  with the 
 coalgebra structure:
\begin{align*}
\Delta(g)=g\otimes g &, \qquad \Delta(h)=1\otimes h+h\otimes g,\\
\epsilon(g)=1 &, \qquad \epsilon(h)=0.
\end{align*}
Moreover, the quotient $\overline{\mathcal{SH}_{n}}:=\mathcal{SH}_{n}/(g^{n}-1,~h^{n})$ is an $\mathbb{N}$-graded pointed Hopf algebra.
\end{proposition}

\begin{proof}
It follows from \cite[Proposition 4.7]{R1977}.
\end{proof}

\begin{remark}
\noindent
$\mathcal{SH}_n$ is infinite-dimensional. 
Assume $q\in \K$ such that $q$ is a primitive $n$-th root of $1$. If $gh=qhg$, then $\mathcal{SH}_{i,n-i}(h,g)=\binom{n}{i}_{q}h^{i}g^{n-i}=0$ for all $1\le i\le n-1$. Thus the Sweedler algebra $H_{4}$ and the Taft algebra $H_{n,q}$ are quotient Hopf algebras of $\mathcal{SH}_n$. If char $\K=0$, then $\overline{\mathcal{SH}_2}$ is the Sweedler algebra $H_{4}$. If char $\K=2$, then $\overline{\mathcal{SH}_2}$ is the group algebra of the Klein group. 

Moreover, Radford studied a $27$-dimensional quotient Hopf algebra of $\overline{\mathcal{SH}_{3}}$ and a $64$-dimensional quotient Hopf algebra of $\overline{\mathcal{SH}_{4}}$.
\end{remark}

\begin{appendix}
\section{Examples of Theorem \ref{thm:noncom-binom-coeff-LSbasis}}
\label{Appendix A}

We rewrite the following shuffle type polynomials $\mathcal{SH}_{k,n-k}(y,x)$ ($4\le n\le 7$) in terms of the Lyndon-Shirshov basis. For the sake of notations of commutators,  we write $E_{1}$ for $x$ and $E_{2}$ for $y$ when commutators appear. 
\begin{example}\label{Example 6.5}
If $n=5$, then
\begin{align*}
& \mathcal{SH}_{0,5}=x^{5}, \\
&\mathcal{SH}_{1,4}=yx^{4}+xyx^{3}+x^{2}yx^{2}+x^{3}yx+x^{4}y,\\
&\mathcal{SH}_{2,3}=y^{2}x^{3}+yxyx^{2}+yx^{2}yx+yx^{3}y+xy^{2}x^{2}+xyxyx+xyx^{2}y+x^{2}y^{2}x+x^{2}yxy+x^{3}y^{2},\\
&\mathcal{SH}_{3,2}=y^{3}x^{2}+y^{2}xyx+y^{2}x^{2}y+yxy^{2}x+yxyxy+yx^{2}y^{2}+xy^{3}x+xy^{2}xy+xyxy^{2}+x^{2}y^{3},\\
&\mathcal{SH}_{4,1}=y^{4}x+y^{3}xy+y^{2}xy^{2}+yxy^{3}+xy^{4},\\
&\mathcal{SH}_{5,0}=y^{5}. 
\end{align*}
Thus
\begin{align*}
\mathcal{SH}_{0,5} =& E_{1}^{5},\\
\mathcal{SH}_{1,4} =& 5E_{2}E_{1}^{4}+10E_{12}E_{1}^{3}+10E_{112}E_{1}^{2}+5E_{1112}E_{1}+E_{11112},\\
\mathcal{SH}_{2,3} =& 10E_{2}^{2}E_{1}^{3}+30E_{2}E_{12}E_{1}^{2}+20E_{2}E_{112}E_{1}+5E_{2}E_{1112}+10E_{122}E_{1}^{2}+15E_{12}^{2}E_{1}+10E_{12}E_{112}+\\
& 5E_{1122}E_{1}+3E_{11212}+E_{11122},\\
\mathcal{SH}_{3,2} =& 10E_{2}^{3}E_{1}^{2}+30E_{2}^{2}E_{12}E_{1}+10E_{2}^{2}E_{112}+20E_{2}E_{122}E_{1}+15E_{2}E_{12}^{2}+5E_{2}E_{1122}+5E_{1222}E_{1}+\\
& 10E_{122}E_{12}+4E_{12122}+E_{11222},\\
\mathcal{SH}_{4,1} =& 5E_{2}^{4}E_{1}+10E_{2}^{3}E_{12}+10E_{2}^{2}E_{122}+5E_{2}E_{1222}+E_{12222},\\
\mathcal{SH}_{5,0}=& E_{2}^{5}. 
\end{align*}

\end{example}
 
\begin{example} 
 For $n=6$, we have:
\begin{align*}
\mathcal{SH}_{0,6} =& E_{1}^{6},\\
\mathcal{SH}_{1,5} =& 6E_{2}E_{1}^{5}+15E_{12}E_{1}^{4}+20E_{112}E_{1}^{3}+15E_{1112}E_{1}^{2}+6E_{11112}E_{1}+E_{111112},\\ 
\mathcal{SH}_{2,4} =& 15E_{2}^{2}E_{1}^{4}+60E_{2}E_{12}E_{1}^{3}+60E_{2}E_{112}E_{1}^{2}+30E_{2}E_{1112}E_{1}+6E_{2}E_{11112}+20E_{122}E_{1}^{3}+45E_{12}^{2}E_{1}^{2}+\\
& 60E_{12}E_{112}E_{1}+15E_{12}E_{1112}+15E_{1122}E_{1}^{2}+18E_{11212}E_{1}+10E_{112}^{2}+6E_{11122}E_{1}+3E_{111212}+\\ & E_{111122},\\
\mathcal{SH}_{3,3} =& 20E_{2}^{3}E_{1}^{3}+90E_{2}^{2}E_{12}E_{1}^{2}+60E_{2}^{2}E_{112}E_{1}+15E_{2}^{2}E_{1112}+60E_{2}E_{122}E_{1}^{2}+90E_{2}E_{12}^{2}E_{1}+60E_{2}E_{12}\\
& E_{112}+30E_{2}E_{1122}E_{1}+18E_{2}E_{11212}+6E_{2}E_{11122}+15E_{1222}E_{1}^{2}+60E_{122}E_{12}E_{1}+20E_{122}E_{112}+\\
& 24E_{12122}E_{1}+15E_{12}^{3}+15E_{12}E_{1122}+6E_{11222}E_{1}+10E_{112212}+4E_{112122}+E_{111222},\\
\mathcal{SH}_{4,2} = &15E_{2}^{4}E_{1}^{2}+60E_{2}^{3}E_{12}E_{1}+20E_{2}^{3}E_{112}+60E_{2}^{2}E_{122}E_{1}+45E_{2}^{2}E_{12}^{2}+15E_{2}^{2}E_{1122}+30E_{2}E_{1222}E_{1}+\\
& 60E_{2}E_{122}E_{12}+24E_{2}E_{12122}+6E_{2}E_{11222}+6E_{12222}E_{1}+15E_{1222}E_{12}+10E_{122}^{2}+5E_{121222}+\\ & E_{112222},\\
\mathcal{SH}_{5,1} =& 6E_{2}^{5}E_{1}+15E_{2}^{4}E_{12}+20E_{2}^{3}E_{122}+15E_{2}^{2}E_{1222}+6E_{2}E_{12222}+E_{122222},\\
\mathcal{SH}_{6,0} =& E_{2}^{6}.
\end{align*}

For $\alpha\in \mathcal{L}$, $C_{E_{\alpha}}$ has effects on the  coefficients of the products which contain $E_{\alpha}$,  see Formula (\ref{j-z formula}). For example, $C_{E_{11212}}=3$ and $C_{E_{12122}}=4$. It is clear that $C_{E_{11212}E_{1}}=18=\frac{6!}{5!1!}C_{E_{11212}}$ and $C_{E_{2}E_{12122}}=24=\frac{6!}{5!1!}C_{E_{12122}}$.
\end{example}

\begin{example}
For $n=7$, we obtain:
\begin{align*}
\mathcal{SH}_{0,7} =& E_{1}^{7}.\\
\mathcal{SH}_{1,6} =& 7E_{2}E_{1}^{6}+21E_{12}E_{1}^{5}+35E_{112}E_{1}^{4}+35E_{1112}E_{1}^{3}+21E_{11112}E_{1}^{2}+7E_{111112}E_{1}+E_{1111112},\\
\mathcal{SH}_{2,5} =& 21E_{2}^{2}E_{1}^{5}+105E_{2}E_{12}E_{1}^{4}+140E_{2}E_{112}E_{1}^{3}+105E_{2}E_{1112}E_{1}^{2}+42E_{2}E_{11112}E_{1}+7E_{2}E_{111112}+\\
& 35E_{122}E_{1}^{4}+105E_{12}^{2}E_{1}^{3}+210E_{12}E_{112}E_{1}^{2}+105E_{12}E_{1112}E_{1}+21E_{12}E_{11112}+35E_{1122}E_{1}^{3}+\\
& 63E_{11212}E_{1}^{2}+70E_{112}^{2}E_{1}+35E_{112}E_{1112}+21E_{11122}E_{1}^{2}+21E_{111212}E_{1}+10E_{1112112}+\\
& 7E_{111122}E_{1}+3E_{1111212}+E_{1111122},\\
\mathcal{SH}_{3,4} =& 35E_{2}^{3}E_{1}^{4}+210E_{2}^{2}E_{12}E_{1}^{3}+210E_{2}^{2}E_{112}E_{1}^{2}+105E_{2}^{2}E_{1112}E_{1}+21E_{2}^{2}E_{11112}+140E_{2}E_{122}E_{1}^{3}+\\
& 315E_{2}E_{12}^{2}E_{1}^{2}+420E_{2}E_{12}E_{112}E_{1}+105E_{2}E_{12}E_{1112}+105E_{2}E_{1122}E_{1}^{2}+126E_{2}E_{11212}E_{1}+\\
& 70E_{2}E_{112}^{2}+42E_{2}E_{11122}E_{1}+21E_{2}E_{111212}+7E_{2}E_{111122}+35E_{1222}E_{1}^{3}+210E_{122}E_{12}E_{1}^{2}+\\
& 140E_{122}E_{112}E_{1}+35E_{122}E_{1112}+84E_{12122}E_{1}^{2}+105E_{12}^{3}E_{1}+105E_{12}^{2}E_{112}+105E_{12}E_{1122}E_{1}+\\
& 63E_{12}E_{11212}+21E_{12}E_{11122}+21E_{11222}E_{1}^{2}+70E_{112212}E_{1}+35E_{1122}E_{112}+28E_{112122}E_{1}+\\
& 15E_{1121212}+15E_{1121122}+7E_{111222}E_{1}+10E_{1112212}+4E_{1112122}+E_{1111222},\\
\mathcal{SH}_{4,3} =& 35E_{2}^{4}E_{1}^{3}+210E_{2}^{3}E_{12}E_{1}^{2}+140E_{2}^{3}E_{112}E_{1}+35E_{2}^{3}E_{1112}+210E_{2}^{2}E_{122}E_{1}^{2}+315E_{2}^{2}E_{12}^{2}E_{1}+\\
& 210E_{2}^{2}E_{12}E_{112}+105E_{2}^{2}E_{1122}E_{1}+63E_{2}^{2}E_{11212}+21E_{2}^{2}E_{11122}+105E_{2}E_{1222}E_{1}^{2}+\\
& 420E_{2}E_{122}E_{12}E_{1}+140E_{2}E_{122}E_{112}+168E_{2}E_{12122}E_{1}+105E_{2}E_{12}^{3}+105E_{2}E_{12}E_{1122}+\\
& 42E_{2}E_{11222}E_{1}+70E_{2}E_{112212}+28E_{2}E_{112122}+7E_{2}E_{111222}+21E_{12222}E_{1}^{2}+105E_{1222}E_{12}E_{1}+\\
& 35E_{1222}E_{112}+70E_{122}^{2}E_{1}+105E_{122}E_{12}^{2}+35E_{122}E_{1122}+35E_{121222}E_{1}+84E_{12122}E_{12}+\\
& 24E_{1212122}+21E_{12}E_{11222}+7E_{112222}E_{1}+15E_{1122212}+10_{1122122}+5E_{1121222}+E_{1112222},\\
\mathcal{SH}_{5,2} =& 21E_{2}^{5}E_{1}^{2}+105E_{2}^{4}E_{12}E_{1}+35E_{2}^{4}E_{112}+140E_{2}^{3}E_{122}E_{1}+105E_{2}^{3}E_{12}^{2}+35E_{2}^{3}E_{1122}+\\
& 105E_{2}^{2}E_{1222}E_{1}+210E_{2}^{2}E_{122}E_{12}+84E_{2}^{2}E_{12122}+21E_{2}^{2}E_{11222}+42E_{2}E_{12222}E_{1}+\\
& 105E_{2}E_{1222}E_{12}+70E_{2}E_{122}^{2}+35E_{2}E_{121222}+7E_{2}E_{112222}+7E_{122222}E_{1}+21E_{12222}E_{12}+\\
& 35E_{1222}E_{122}+15E_{1221222}+6E_{1212222}+E_{1122222},\\
\mathcal{SH}_{6,1} =& 7E_{2}^{6}E_{1}+21E_{4}^{5}E_{12}+35E_{2}^{4}E_{122}+35E_{2}^{3}E_{1222}+21E_{2}^{2}E_{12222}+7E_{2}E_{122222}+E_{1222222},\\
\mathcal{SH}_{7,0} =& E_{2}^{7}.
\end{align*}
\end{example}
\end{appendix}

\section*{\textbf{ACKNOWLEGEMENT}}
The first author thanks the China Scholarship Council (No. 201906140164) and the BOF of UHasselt for their financial support during his research study at the University of Hasselt.

\bibliographystyle{plain}

\end{document}